\documentclass[11pt]{article}
\usepackage{amssymb,amsmath,amsthm}
\newcommand{\EE}{\mathcal{E}}
\newcommand{\FF}{\mathcal{F}}
\newcommand{\OO}{\mathcal{O}}
\newcommand{\VV}{\mathcal{V}}
\newcommand{\TT}{\mathcal{T}}
\newcommand{\UU}{\mathcal{U}}

\newcommand{\N}{\mathbb{N}}
\renewcommand{\O}{\mathbb{O}}
\renewcommand{\P}{\mathbb{P}}
\newcommand{\R}{\mathbb{R}}
\newcommand{\T}{\mathbb{T}}
\newcommand{\Z}{\mathbb{Z}}

\newcommand{\dd}{\,{\rm d}}

\renewcommand{\div}{\mathop{\mathrm{div}}}
\newcommand{\curl}{\mathop{\mathrm{curl}}}

\renewcommand{\:}{\thinspace :}
\newcommand{\meas}{\mathrm{meas}}
\newcommand{\loc}{\mathrm{loc}}

\newcommand{\DS}{\displaystyle}
\newcommand{\1}{\mathbf{1}}
\newcommand{\tsp}{\mathrm{t}}

\newcommand{\BUC}{\mathrm{BUC}}
\newcommand{\BMO}{\mathrm{BMO}}

\newtheorem{theorem}{Theorem}[section]
\newtheorem{definition}[theorem]{Definition}
\newtheorem{proposition}[theorem]{Proposition}
\newtheorem{lemma}[theorem]{Lemma}
\newtheorem{corollary}[theorem]{Corollary}

\theoremstyle{definition}
\newtheorem{remark}[theorem]{Remark}

\newcommand{\QED}{\mbox{}\hfill$\Box$}
\numberwithin{equation}{section}

\begin{document}

\title{Energy bounds for the two-dimensional Navier-Stokes 
equations in an infinite cylinder}

\author{
\null\\
{\bf Thierry Gallay}\\ 
Institut Fourier\\
Universit\'e de Grenoble 1\\
BP 74\\
38402 Saint-Martin-d'H\`eres, France\\
{\tt Thierry.Gallay@ujf-grenoble.fr}
\and
\\
{\bf Sini\v{s}a Slijep\v{c}evi\'c}\\    
Department of Mathematics\\ 
University of Zagreb\\
Bijeni\v{c}ka 30\\
10000 Zagreb, Croatia\\
{\tt slijepce@math.hr}}

\date{}

\maketitle

\begin{abstract}
We consider the incompressible Navier-Stokes equations in the cylinder
$\R \times \T$, with no exterior forcing, and we investigate the long-time
behavior of solutions arising from merely bounded initial data. Although
we do not know if such solutions stay uniformly bounded for all times, 
we prove that they converge in an appropriate sense to the family of 
spatially homogeneous equilibria as $t \to \infty$. Convergence is 
uniform on compact subdomains, and holds for all times except on 
a sparse subset of the positive real axis. We also improve the known
upper bound on the $L^\infty$ norm of the solutions, although our
results in this direction are not optimal. Our approach is based 
on a detailed study of the local energy dissipation in the system, in
the spirit of a recent work devoted to a class of dissipative partial 
differential equations with a formal gradient structure \cite{GS2}. 
\end{abstract}

\maketitle

\section{Introduction}\label{sec1}

The aim of this paper is to give some insight into the intrinsic
dynamics of the two-dimensional incompressible Navier-Stokes equations
in an unbounded domain. We consider the situation where a viscous
fluid evolves freely without being driven by any external force, so
that the motion originates entirely from the initial data, and we aim
at obtaining general informations on the long-time behavior of the
system. If the fluid fills a bounded domain $\Omega \subset \R^2$ and
satisfies no-slip boundary conditions, it is well-known that the
velocity converges exponentially fast to zero as $t \to \infty$, and
the long-time asymptotics can be accurately described \cite{FS}. In an
unbounded domain such as the whole plane $\Omega = \R^2$, solutions
with finite kinetic energy also converge to the uniform rest state
\cite{Ma}, and the (algebraic) decay rate of the velocity can be
specified under appropriate localization assumptions on the initial
data \cite{Sch,Wi}. Similar results can be obtained for
infinite-energy solutions if the vorticity of the fluid is integrable
\cite{GW}, in which case the velocity field decays to zero roughly
like $|x|^{-1}$ as $|x| \to \infty$.

The problem is far more complicated if we consider the situation 
where the velocity field is merely bounded, or decays very slowly to
zero at infinity. In that case the Cauchy problem for the
Navier-Stokes equations is still globally well-posed \cite{GMS,ST}, but
essentially nothing is known about the long-time behavior of the
solutions. In fact, it is even unclear whether the $L^\infty$ norm of
the velocity field $u(\cdot,t)$ stays bounded for all times. For instance,
if the fluid fills the whole plane $\R^2$, the best estimate we are
aware of is due to Zelik \cite{Ze}, and takes the form
$\|u(\cdot,t)\|_{L^\infty} \le C(1+t)^2$, but there is no reason to
believe that such a bound is sharp. Although the question may look
rather academic at first sight, we believe that it is important to
understand the behavior of solutions that are just bounded, 
because they may exhibit a nontrivial dynamics that is 
created by the equation itself, through the initial data, and 
does not result from an artificial exterior forcing. An intimately 
related problem is to understand the dynamics of the Navier-Stokes
equations in large but bounded domains, and in particular to derive 
{\em uniformly local} energy estimates that are independent of 
the size of the domain, or of the total kinetic energy. 

As a first step in this direction, following Afendikov and Mielke
\cite{AM}, we study in the present paper the situation where the 
fluid velocity $u(x,t) \in \R^2$ and the pressure $p(x,t) \in \R$ are 
{\em periodic} in one space direction. As is discussed in \cite{AM}, 
this assumption considerably simplifies the dynamics of the Navier-Stokes 
equations, but the periodic setting still includes interesting and 
nontrivial examples, such as Kolmogorov flows. We thus consider 
the system
\begin{equation}\label{NSeq}
  \partial_t u + (u\cdot \nabla)u \,=\, \Delta u-\nabla p~, 
  \qquad \div u \,=\, 0~,
\end{equation}
in the two-dimensional cylinder $\O = \R\times\T$, where $\T = \R/\Z$.
If $x = (x_1,x_2) \in \O$, the unbounded coordinate $x_1 \in \R$ will
be referred to as the ``horizontal variable'' and $x_2 \in \T$ will be
the ``vertical variable'', although no gravity is included in our
model. For simplicity, all physical parameters, such as the kinematic
viscosity and the fluid density, have been eliminated by rescaling.
Eq.~\eqref{NSeq} is the system studied by Afendikov and Mielke in the
particular case where no exterior force drives the fluid. Our goal
is to sharpen the conclusions of \cite{AM} and to obtain more
precise information on the long-time behavior of the solutions in that
particular situation.

Since we want to consider solutions of \eqref{NSeq} that are merely
bounded, we suppose that the velocity field $u(\cdot,t)$ belongs for
each $t \ge 0$ to the Banach space $\BUC(\O)$ of all bounded and
uniformly continuous functions $u : \O \to \R^2$ equipped with the
supremum norm
\[
  ||u||_{L^\infty(\O)} \,=\, \sup_{x\in \O}|u(x)|~, \qquad \hbox{where}
  \quad |u| = (u_1^2 + u_2^2)^{1/2}~. 
\]
Given a velocity field $u \in \BUC(\O)$, the associated pressure $p$ 
is determined by solving the elliptic equation $-\Delta p = 
\div((u\cdot\nabla)u)$ with appropriate conditions at infinity. 
As is explained in \cite{AM}, it is very important to specify the 
behavior of the pressure, because different choices lead 
to genuinely different dynamics. For flows that are not driven by a 
pressure gradient at infinity, the standard choice is 
\begin{equation}\label{pdef}
  p \,=\, \sum_{k,\ell=1}^2 R_k R_\ell u_k u_\ell~, 
\end{equation}
see \cite{GMS,Ka,AM}, where $R_1, R_2$ are the Riesz transforms 
in the cylinder $\O = \R \times \T$. For instance, if $u \in
\BUC(\O)$, the pressure $p$ defined by \eqref{pdef} belongs to 
$\BMO(\O)$, the space of all functions with bounded mean oscillation
on $\O$. Taking \eqref{pdef} into account, the Navier-Stokes 
equations \eqref{NSeq} can be written in the compact form 
$\partial_t u + \P (u\cdot\nabla)u = \Delta u$, where $\P$ 
is the Leray-Hopf projection defined by its matrix elements 
$\P_{jk} = \delta_{jk} + R_j R_k$, $1 \le j,k \le 2$.  

The following result summarizes the main conclusions of 
\cite{AM} for flows without exterior forcing\:

\begin{theorem}\label{thm1} {\bf \cite{AM}}
For any initial data $u_0 \in \BUC(\O)$ with $\div u_0 = 0$, 
the Navier-Stokes equations \eqref{NSeq}, \eqref{pdef} have 
a unique global solution $u \in C^0([0,+\infty),\BUC(\O))$ such 
that $u(0) = u_0$. Moreover, there exists $C > 0$ (depending 
on $u_0$) such that 
\begin{equation}\label{ubound1}
  \|u(\cdot,t)\|_{L^\infty(\O)} \,\le\, C(1+t)^{1/2}~, \qquad 
  \hbox{for all }~ t \ge 0~.
\end{equation}
\end{theorem}

An interesting open question is whether the solutions of \eqref{NSeq}, 
\eqref{pdef} given by Theorem~\ref{thm1} stay uniformly bounded for
all times. In this direction, we just mention the following improvement 
of estimate \eqref{ubound1}, which will come as a byproduct of our 
analysis. 

\begin{theorem}\label{thm2} Under the assumptions of 
Theorem~\ref{thm1}, there exists a constant $C > 0$ (depending 
on the initial data $u_0$) such that 
\begin{equation}\label{ubound2}
  \|u(\cdot,t)\|_{L^\infty(\O)} \,\le\, C(1+t)^{1/6}~, \qquad 
  \hbox{for all }~ t \ge 0~.
\end{equation}
\end{theorem}

As is explained in Section~\ref{sec7} below, there are good reasons to
believe that the bound \eqref{ubound2} is not sharp either for large
times, but it is not clear what the optimal result should be.  This
question will not be addressed further here, and we hope to come back
to it in a future work. For the time being, our main purpose is to
show that it is possible to obtain rather precise information on the
long-time dynamics of equation \eqref{NSeq} even if uniform bounds 
on the solutions are not known a priori. To formulate our results, 
it is convenient to assume that the {\em mean horizontal flow} vanishes 
identically\:
\begin{equation}\label{czero}
  \langle u_1\rangle (x_1,t) \,:=\, \int_\T u_1(x_1,x_2,t)\dd x_2 
  \,=\, 0~, \qquad \hbox{for all }~ x_1 \in \R\,,~t \ge 0~.
\end{equation}
This, however, does not restrict the generality, as can be seen by the
following argument. Given a solution of \eqref{NSeq}, we define $m_1 =
\langle u_1\rangle$ and $a = \langle u_1^2 + p\rangle$, where
the brackets $\langle \cdot\rangle$ denote the vertical average, as in
\eqref{czero}. Using the divergence-free condition and integrating by
parts, it is easy to verify that $\partial_{x_1} m_1 = 0$, so that $m_1$
is a function of $t$ only. On the other hand, using the first equation
in \eqref{NSeq} we find $\partial_t m_1 + \partial_{x_1}a = 0$, thus 
$\partial_{x_1}a$ is also a function of $t$ only. If we assume that $u
\in \BUC(\O)$ and $p \in \BMO(\O)$, as in Theorem~\ref{thm1}, this
implies that $\partial_{x_1}a = 0$, hence $m_1 = \langle u_1\rangle$ is
a constant that can be computed from the initial data. We now define
\[
  \begin{pmatrix} \tilde u_1(x_1,x_2,t) \\ \tilde u_2(x_1,x_2,t)
  \end{pmatrix} \,=\, \begin{pmatrix} u_1(x_1+m_1t,x_2,t) - m_1\\ 
  \!u_2(x_1+m_1t,x_2,t)\phantom{ - m_1}\end{pmatrix}~, \qquad 
  \tilde p(x_1,x_2,t) \,=\, p(x_1+m_1t,x_2,t)~, 
\]
for all $x \in \O$ and $t \ge 0$. By Galilean invariance, it is clear 
that $\tilde u, \tilde p$ solve the Navier-Stokes equation \eqref{NSeq}, 
and by construction the mean horizontal flow $\langle \tilde u_1\rangle$ 
vanishes identically. 

The following theorem collects a few typical consequences 
of our main results, which will be formulated more precisely 
and in a greater generality in the subsequent sections. 

\begin{theorem}\label{thm3} Let $u$ be a solution of the Navier-Stokes 
equations given by Theorem~\ref{thm1} and satisfying \eqref{czero}. 
Then the following estimates hold\:\\[1mm]
1) There exists $C > 0$ (depending only on $u_0$) such that, 
for all $T > 0$, 
\begin{equation}\label{ubound3}
  \sup_{x_1 \in \R}\, \int_0^T \!\!\int_\T |u(x_1,x_2,t)|^2 \dd x_2 \dd t 
  \,\le\, C T~.
\end{equation}
2) There exists $C > 0$ (depending only on $u_0$) such that, 
for all $T > 0$ and all $R > 0$, 
\begin{equation}\label{ubound4}
  \int_{-R}^R \int_\T |u(x_1,x_2,T)|^2 \dd x_2 \dd x_1 + 
  \int_0^T \!\!\int_{-R}^R \int_\T |\nabla u(x_1,x_2,t)|^2 \dd x_2 
  \dd x_1 \dd t \,\le\, C(R + T^{1/2})~.
\end{equation}
3) For all $\epsilon > 0$ and all $R > 0$, we have
\begin{equation}\label{ubound5}
  \frac{1}{T}\,\meas\biggl\{t \in [0,T]\,\Big|\, \inf_{m \in \R}
  \,\sup_{|x_1|\le R}\,\sup_{x_2 \in \T} |u(x_1,x_2,t) - (0,m)^\tsp| 
  \ge \epsilon\biggr\}\,\xrightarrow[T \to \infty]{}\, 0~.
\end{equation}
\end{theorem}

Before giving an idea of our general strategy, we briefly comment on
the results summarized in Theorem~\ref{thm3}. Estimate \eqref{ubound3}
shows that the ``kinetic energy'' $\frac12 \langle |u|^2 \rangle$,
computed at each point $x_1 \in \R$, behaves like a bounded quantity
when averaged over time. This already indicates that a bound like
\eqref{ubound1} or \eqref{ubound2} is necessarily pessimistic\: if the
quantity $\langle |u|^2 \rangle$ is not uniformly bounded in time, it
can reach large values only on a relatively small subset of the time
interval, otherwise \eqref{ubound3} would give a contradiction. In
particular $\langle |u|^2 \rangle$ cannot increase to infinity as $t
\to \infty$ at any fixed point $x_1 \in \R$. Estimate \eqref{ubound4}
contains even more information, and for simplicity we only comment on
the particular case where $R = T^{1/2}$, which is especially
instructive. We first learn from \eqref{ubound4} that the energy
$\frac12\langle |u|^2 \rangle$ computed at any time $T > 0$ behaves
like a bounded quantity when averaged over an interval of size
$T^{1/2}$ in the horizontal variable. As before, this indicates that,
for any fixed $T > 0$, the quantity $\langle |u|^2 \rangle$ can reach
large values only on relatively small spatial subdomains. Next,
Eq.~\eqref{ubound4} shows that the energy dissipation $\langle |\nabla
u|^2 \rangle$ converges to zero as $T \to \infty$ when averaged over a
horizontal interval of size $T^{1/2}$ and a time interval of size
$T$. This information is new and valuable, even for solutions for
which a uniform upper bound is known a priori. As a consequence of
\eqref{ubound4}, we immediately see that the only time-independent or
time-periodic solutions of the Navier-Stokes equations \eqref{NSeq},
\eqref{pdef} in $\BUC(\O)$ are spatially homogeneous equilibria of the
form $u = (m_1,m_2)^\tsp$, where $m_1, m_2 \in \R$ (of course, under
assumption \eqref{czero}, we have $m_1 = 0$). Finally, it follows from
\eqref{ubound4} that any solution of \eqref{NSeq}, \eqref{pdef} in
$\BUC(\O)$ converges uniformly on compact subdomains to this family of
equilibria, except perhaps on a sparse subset of the time axis. A
simple version of this last statement is given in \eqref{ubound5}, and
we refer to Corollary~\ref{c:time2} below for a more precise and
quantitative estimate.

Our analysis of the dynamics of the two-dimensional Navier-Stokes
equations \eqref{NSeq} is based on the following simple ideas.  If
$u(x,t), p(x,t)$ is a solution of \eqref{NSeq}, \eqref{pdef} given by
Theorem~\ref{thm1}, we introduce the {\em energy density} $e = \langle
\frac12 |u|^2\rangle + 1$, the {\em inviscid flux} $h = \langle
(\frac12 |u|^2 + p)u_1\rangle$, and the {\em energy dissipation rate}
$d = \langle \frac12 |\nabla u|^2\rangle$. More explicitly, for all
$x_1 \in \R$ and all $t > 0$, we define
\begin{align}\label{edef}
  e(x_1,t) \,&=\, \frac12 \int_\T |u(x_1,x_2,t)|^2 \dd x_2 + 1~, \\
  \label{bdef}
  h(x_1,t) \,&=\, \int_\T \Bigl(\frac12 |u(x_1,x_2,t)|^2 + 
  p(x_1,x_2,t)\Bigr)u_1(x_1,x_2,t)\dd x_2~, \\ \label{ddef}
  d(x_1,t) \,&=\, \int_\T |\nabla u(x_1,x_2,t)|^2 \dd x_2~,
\end{align}
where $|\nabla u|^2 = |\partial_1 u_1|^2 + |\partial_2 u_1|^2 +
|\partial_1 u_2|^2 + |\partial_2 u_2|^2$.  Here and in the sequel, we
denote $\partial_1 = \partial_{x_1}$ and $\partial_2 = \partial_{x_2}$
for simplicity. A straightforward calculation then shows that the
energy density satisfies $\partial_t e + \partial_1 h \,=\, 
\partial_1^2 e - d$ for all $x_1 \in \R$ and $t > 0$. In particular,
if we introduce the {\em total energy flux} $f = \partial_1 e - h$, 
we arrive at the energy balance equation
\begin{equation}\label{ebal}
  \partial_t e(x_1,t) \,=\, \partial_1 f(x_1,t) - d(x_1,t)~,
  \qquad x_1 \in \R~, \quad t > 0~,
\end{equation}
which is the starting point of our approach. 

At this point, we would like to mention that \eqref{ddef} is {\em not}
the usual definition of the energy dissipation rate that can be found
in textbooks of Fluid Mechanics. Indeed, energy is dissipated in
viscous fluids due to internal friction, and the rate of dissipation
is therefore proportional to $|D(u)|^2$ instead of $|\nabla u|^2$,
where $D(u) = \nabla u + (\nabla u)^\tsp$ is the strain rate tensor.
Albeit less natural from a physical point of view, the definition 
\eqref{ddef} seems nevertheless more convenient for our purposes, 
because the energy dissipation rates then controls all first-order 
derivatives of $u$. 

To exploit \eqref{ebal}, it is necessary to bound the energy flux $f$
in terms of $e$ and $d$. In Section~\ref{sec3} below we show that, for
any $t_0 > 0$, there exists a constant $C > 0$ such that 
\begin{equation}\label{hbound} 
  h(x_1,t)^2 \,\le\, C e(x_1,t) d(x_1,t)~, \qquad x_1 \in \R~, \quad 
  t \ge t_0~. 
\end{equation}
This simple bound is obtained under the assumption that $e \ge 1$, and 
this is why we added a constant to the kinetic energy in the definition 
\eqref{edef}. On the other hand, using \eqref{edef}, \eqref{ddef}
and the Cauchy-Schwarz inequality, we easily obtain $(\partial_1 e)^2 
\le 2ed$. Summarizing, there exists $\beta > 0$ such that
\begin{equation}\label{fbound} 
  f(x_1,t)^2 \,\le\, \beta e(x_1,t) d(x_1,t)~, \qquad  x_1 \in \R~, 
  \quad t \ge t_0~.
\end{equation}
This inequality is very important, because it shows that energy is 
necessarily dissipated in the system as soon as the flux $f$ 
is nonzero. More precisely, if we have an upper bound on the 
energy density $e$, then \eqref{fbound} allows to quantify how
much energy is dissipated during transport. 

In a recent paper \cite{GS2}, we introduced the notion of an {\em
  extended dissipative system} in a rather general framework.  Roughly
speaking, this is a system in which one can define an energy density
$e$, an energy flux $f$, and an energy dissipation rate $d$ satisfying
\eqref{fbound} and such that the energy balance \eqref{ebal} holds for
all solutions, see Section~\ref{sec3} below for more details. Under
these assumptions, we established in \cite{GS2} a few general results
which impose rather severe constraints to the dynamics of the
system. For instance, nontrivial time-periodic orbits cannot exist,
and any global solution with uniformly bounded energy density
converges, in a suitable localized and averaged sense, to the set of
equilibria.  Unfortunately, the results of \cite{GS2} do not apply
directly to the Navier-Stokes equations \eqref{NSeq}, \eqref{pdef},
because we do not know a priori if the energy density \eqref{edef}
stays uniformly bounded for all times, see the discussion near
Theorem~\ref{thm2} above. The purpose of the present paper is
precisely to extend the techniques developed in \cite{GS2} so as to
cover the important case of the Navier-Stokes equations. The main new
ingredient is the estimate $(\partial_1 e)^2 \le 2ed$, which holds in
the present case but was not included in our abstract definition of an
extended dissipative system (because it is not satisfied in some other
important examples).  When combined with \eqref{fbound}, this estimate
allows to obtain convergence results that are exactly as accurate as
those derived in \cite{GS2} for uniformly bounded solutions of general
extended dissipative systems.

The rest of this paper is organized as follows. In Section~\ref{sec2},
we briefly recall what is known about the Cauchy problem for the
Navier-Stokes equations \eqref{NSeq}, \eqref{pdef} in the space
$\BUC(\O)$, and we give a short proof of Theorem~\ref{thm1}. In
Section~\ref{sec3}, we show that the Navier-Stokes equations define an
extended dissipative system in the sense of \cite{GS2}, and we
establish the crucial estimate \eqref{fbound} using a uniform bound on
the vorticity $\omega = \partial_1 u_2 - \partial_2 u_1$. The main
part of our analysis begins in Sections~\ref{sec4} and \ref{sec5},
where we obtain (in an abstract framework) accurate estimates on the
integrated energy flux and the integrated energy density. Unlike in
\cite{GS2}, we do not have to assume here that the energy density is
uniformly bounded for all times; nevertheless, we can draw similar
conclusions concerning the long-time behavior of the solutions, some
of which are presented in Section~\ref{sec6}. The (rather delicate)
question of obtaining pointwise estimates on the energy density is
briefly discussed in Section~\ref{sec7}, which contains in particular
a proof of Theorem~\ref{thm2}. Finally, we show in Section~\ref{sec8}
that the solutions of the Navier-Stokes equations \eqref{NSeq},
\eqref{pdef} converge (in an appropriate sense) to spatially
homogeneous equilibria as $t \to \infty$, and we give a proof of
Theorem~\ref{thm3} which includes a much more precise version of
estimate \eqref{ubound5}.  The last section is an appendix which
collects the proofs of some auxiliary results stated in 
Section~\ref{sec2}. 

\medskip\noindent{\bf Acknowledgements.} The research of the
second named author was partially supported by the grant 
No 037-0372791-2803 of the Croatian Ministry of Science.

\section{The Navier-Stokes and vorticity equations in 
$\O = \R \times\T$}\label{sec2}

In this section, we study the Cauchy problem for the Navier-Stokes
equations \eqref{NSeq}, \eqref{pdef} in the cylinder $\O = \R \times\T$, 
and we establish a few general properties of the solutions that 
will be used later on. We do not claim for much originality at this
stage, because the results collected here are essentially taken 
from \cite{AM,GMS,GS2}. 

As was observed in the introduction, the Navier-Stokes equation can be
written in the form $\partial_t u + \P (u\cdot\nabla)u = \Delta u$,
where $\P$ is the Leray-Hopf projection. Given initial data $u_0$, the
corresponding integral equation reads
\begin{equation}\label{NSint}
  u(t) \,=\, e^{t\Delta}u_0 - \int_0^t \nabla\cdot e^{(t-s)\Delta}
  \P (u(s) \otimes u(s))\dd s~, \qquad t \ge 0~,
\end{equation}
where $u(t) = u(\cdot,t)$ and $\nabla\cdot e^{t\Delta}\P(u\otimes v)$ 
is a shorthand notation for the vector with $j^{\mathrm{th}}$ component
\begin{equation}\label{nlinexp}
  \Bigl(\nabla\cdot e^{t\Delta}\P(u\otimes v)\Bigr)_j \,=\, 
  \sum_{k,\ell = 1}^2 \partial_\ell \,e^{t\Delta}\P_{jk}u_\ell v_k~,
  \qquad j = 1,2~.
\end{equation}
It is well known that the heat kernel $e^{t\Delta}$ defines a strongly 
continuous semigroup of contractions in the space $\BUC(\O)$, see
e.g. \cite{ARCD}. Moreover, the following estimate allows to control 
the nonlinear term in \eqref{NSint}\:

\begin{lemma}\label{nlinlem}
There exists a constant $C_0 > 0$ such that, for all $t > 0$ and 
all $u,v \in \BUC(\O)$, one has $\nabla\cdot e^{t\Delta}\P(u\otimes v)
\in \BUC(\O)$ and
\begin{equation}\label{nlinest}
  \|\nabla\cdot e^{t\Delta}\P(u\otimes v)\|_{L^\infty(\O)} \,\le\, 
  \frac{C_0}{\sqrt{t}} \,\|u\|_{L^\infty(\O)}\|v\|_{L^\infty(\O)}~.
\end{equation}
\end{lemma}

For the reader's convenience, we give a short proof of estimate 
\eqref{nlinest} in the Appendix. Using Lemma~\ref{nlinlem}
and a standard fixed point argument, one easily obtains the 
following local existence result. 

\begin{proposition}\label{localex} {\bf \cite{AM,GMS}}
For any initial data $u_0 \in \BUC(\O)$ with $\div u_0 = 0$, there 
exists a time $T > 0$ such that the integral equation \eqref{NSint}
has a unique solution $u \in C^0([0,T],\BUC(\O))$, which satisfies 
$t^{1/2}\partial_j u \in C^0_b((0,T],\BUC(\O))$ for $j = 1,2$. 
\end{proposition}

As in \cite{GMS}, one can also show that the solutions of
\eqref{NSint} are smooth and satisfy \eqref{NSeq}, \eqref{pdef} for
positive times. The proof of Theorem~\ref{localex}, which is
reproduced in the Appendix, gives a local existence time 
of the form $T = \OO(\|u_0 \|_{L^\infty}^{-2})$, so that any upper bound 
on $\|u_0\|_{L^\infty}$ provides a lower bound on $T$. In particular, all
solutions either blow up in finite time in $L^\infty$ norm, or can be
extended to the whole time axis $[0,+\infty)$. To rule out the first
scenario, the most efficient way is to consider the vorticity $\omega
= \curl u = \partial_1 u_2 - \partial_2 u_1$, which is well defined
for positive times by Proposition~\ref{localex} and evolves according
to the advection-diffusion equation
\begin{equation}\label{omeq}
  \partial_t\omega + u\cdot \nabla \omega  \,=\, \Delta \omega~.
\end{equation}
The parabolic maximum principle applies to \eqref{omeq}, hence
$\|\omega(\cdot,t)\|_{L^\infty}$ is a nonincreasing function of 
$t > 0$. This, however, does not imply that the velocity field
$u(\cdot,t)$ stays uniformly bounded for all times, because 
the Biot-Savart law does not allow to control the vertical 
average of $u$ in terms of $\omega$, as we now explain. 

Given any divergence-free velocity field $u \in \BUC(\O)$, we 
decompose $u = \langle u\rangle + \widehat u $, where $\langle 
u\rangle = \int_\T u\dd x_2$ denotes the vertical average of 
$u$. More explicitly, 
\begin{equation}\label{udecomp}
  u(x_1,x_2) \,=\, \begin{pmatrix} m_1(x_1) \\ m_2(x_1) 
  \end{pmatrix} + \begin{pmatrix} \widehat u_1(x_1,x_2) \\
  \widehat u_2(x_1,x_2)\end{pmatrix}~, \qquad 
  (x_1,x_2) \in \R \times \T~,
\end{equation}
where $m_j = \langle u_j\rangle$ for $j = 1,2$. The divergence-free
condition implies that $\div \langle u \rangle = \partial_1 m_1 = 0$
and $\div \widehat u = 0$. In particular, the mean horizontal flow
$m_1$ is a constant which (according to the discussion in the previous
section) can be set to zero without loss of generality. We thus
assume that $(m_1,m_2) = (0,m)$, where $m = \langle u_2\rangle$ is the
mean vertical flow. Taking the curl of \eqref{udecomp}, we also obtain
$\omega = \langle \omega\rangle + \widehat \omega $, where
$\langle\omega\rangle = \partial_1 m$ and $\widehat \omega
= \partial_1 \widehat u _2 - \partial_2 \widehat u _1$. Now, a direct
calculation which can be found in \cite{AM} shows that the oscillating
part $\widehat u$ of the velocity field is entirely determined by the
associated vorticity $\widehat \omega$. More precisely, we have the
Biot-Savart formula\:
\begin{equation}\label{BS}
  \widehat u(x_1,x_2) \,=\, \int_\R\int_\T \nabla^\perp K(x_1-y_1,x_2-y_2)
  \,\widehat \omega(y_1,y_2) \dd y_2 \dd y_1~, 
\end{equation}
where $\nabla^\perp = (-\partial_2,\partial_1)^\tsp$ and $K$ is the 
fundamental solution of the Laplace operator in $\O = \R \times \T$\:
\begin{equation}\label{Kdef}
  K(x_1,x_2) \,=\, \frac{1}{4\pi}\log\Bigl(2\cosh(2\pi x_1) 
  - 2\cos(2\pi x_2)\Bigr)~.
\end{equation}
In contrast, the mean vertical flow $m = \langle u_2\rangle$ cannot be
completely expressed in terms of the vorticity, and we only know that
$\partial_1 m = \langle \omega\rangle$.

For later use, we also give an explicit formula for the pressure
corresponding to the velocity field \eqref{udecomp} with $(m_1,m_2) =
(0,m)$. Note that \eqref{pdef} only defines $p$ up to a constant if $u
\in \BUC(\O)$, and it is necessary to fix that constant if we want 
to control $p$ in a space like $L^\infty(\O)$. In the Appendix,  
we shall show that $p$ can be taken as
\begin{equation}\label{pdef2}
  p \,=\, -u_1^2 - 2 \partial_2 K * (\omega \,u_1)~,
\end{equation}
where $*$ denotes the convolution product in $\O$, see \eqref{BS}. 
We then have the following result, whose proof is also postponed to
the Appendix. 

\begin{lemma}\label{l:bounds}
Assume that $u \in \BUC(\mathbb{O)}$ satisfies $\div u = 0$ and 
$\langle u_1 \rangle = 0$, and let $p$ be defined by \eqref{pdef2} 
where $\omega = \partial_1 u_2 - \partial_2 u_1$. If $\omega \in 
L^\infty(\O)$, we have 
\begin{equation}\label{upest}
  \|\widehat{u}\|_{L^{\infty }(\O)} \,\le\, C_1 \|\omega\|_{L^\infty(\O)}~,
  \qquad \|p\|_{L^{\infty }(\O)} \,\le\, C_2 \|\omega\|_{L^\infty(\O)}^2~,
\end{equation}
where $C_1, C_2$ are positive constants which do not depend on $u$.
Moreover, there exists $C_3 > 0$ such that 
\begin{equation}\label{nablauest}
 \|\nabla u\|_{\BMO(\O)} \,\le\, C_3 \|\omega\|_{L^\infty(\O)}~.
\end{equation}
\end{lemma}

Here $\BMO(\O)$ denotes the space of functions with bounded mean 
oscillation on $\O$, which can be identified with the space of all  
$f \in \BMO(\R^2)$ that are $1$-periodic in the vertical direction,  
see \cite[Chapter IV]{St} for precise definitions. If $f \in \BMO(\O)$, 
then $\|f\|_{\BMO(\O)}  = \|\tilde f\|_{\BMO(\R^2)}$ where $\tilde f$ 
denotes the periodic extension of $f$ to the whole plane $\R^2$.  

We now return to the solutions of \eqref{NSeq} given by
Proposition~\ref{localex} and show that they cannot blow up in finite
time. Take $u_0 \in \BUC(\O)$ such that $\div u_0 = 0$ and $\langle
(u_0)_1\rangle = 0$, and let $T_* > 0$ be the maximal existence time
of the solution of \eqref{NSeq}, \eqref{pdef} with initial data $u_0$. 
As was already observed, the vorticity $\omega(\cdot,t)$ is uniformly 
bounded for all $t \in [t_0,T_*)$, if $t_0 > 0$.  Lemma~\ref{l:bounds} then 
shows that the horizontal speed $u_1 = \widehat u_1$ and the oscillating 
part $\widehat u_2$ of the vertical speed are uniformly bounded for all 
$t \in [0,T_*)$. Thus we only need to estimate the mean vertical flow 
$m(x_1,t)$, which satisfies the simple
equation
\begin{equation}\label{meq}   
  \partial_1 m + \partial_1 \langle \widehat u_1 \widehat u_2\rangle
  \,=\, \partial_1^2 m~, \qquad x_1 \in \R~, \quad 0 < t < T_*~.
\end{equation}
Indeed, using the identity $(u\cdot\nabla)u = \frac12 \nabla|u|^2 + 
u^\perp \omega$ and averaging over $x_2 \in \T$ the evolution equation 
for $u_2$ in \eqref{NSeq}, we easily obtain the equation $\partial_t m  + 
\langle u_1 \omega\rangle = \partial_1^2 m$, which is equivalent to 
\eqref{meq} since $\langle u_1 \omega\rangle = \langle \widehat u_1 
\widehat \omega\rangle = \langle \widehat u_1 \partial_1 \widehat u_2 
\rangle - \langle \widehat u_1 \partial_2 \widehat u_1 \rangle  
= \partial_1 \langle \widehat u_1 \widehat u_2\rangle$. Now, the 
integral equation corresponding to \eqref{meq} is
\[
  m(t) \,=\, e^{t\partial_1^2}m(0) - \int_0^t \partial_1 e^{(t-s)\partial_1^2}
  \langle \widehat u_1(s) \widehat u_2(s)\rangle \dd s~, \qquad 
  0 < t < T_*~,
\]
where $e^{t\partial_1^2}$ denotes the heat semigroup on $\R$. Since 
there exists $M > 0$ such that $\|\widehat u(t)\|_{L^\infty} \le M$
for all $t \in [0,T_*)$, we have
\[
  \|m(t)\|_{L^\infty} \,\le\, \|m(0)\|_{L^\infty} + \int_0^t 
  \frac{M^2}{\sqrt{\pi(t-s)}}\dd s \,=\, \|m(0)\|_{L^\infty}
  + \frac{2M^2\,t^{1/2}}{\sqrt{\pi}}~, \qquad 0 \le t < T_*~.
\]
This shows that $T_* = +\infty$, and that estimate \eqref{ubound1} 
holds for some $C > 0$. The proof of Theorem~\ref{thm1} is thus
complete. \QED

\section{The Navier-Stokes equation as an extended 
dissipative system}\label{sec3}

In a previous work \cite{GS2}, we introduced the notion of an extended
dissipative system, which is an abstract framework describing the
essential properties of an important class of dissipative partial
differential equations on unbounded domains. In this section, we show
that the Navier-Stokes equation \eqref{NSeq} belongs to that class, so
that interesting conclusions can be drawn for solutions that stay
uniformly bounded for all times. However, as was discussed in the
introduction, it is still unclear whether all solutions of
\eqref{NSeq} stay bounded, and without uniform bound on the energy the
results of \cite{GS2} do not give much information on the long-time
behavior.

We first recall the main definitions. If $X$ is a metrizable topological
space, we say that a family $(\Phi(t))_{t\ge0}$ of continuous maps 
$\Phi(t) : X \to X$ is a {\em continuous semiflow} on $X$ if\\[1mm]
\null\hskip 8pt $\bullet$ $\Phi(0) = \1$ (the identity map);\\[1mm]
\null\hskip 8pt $\bullet$ $\Phi(t_1+t_2) = \Phi(t_1)\circ\Phi(t_2)$ 
for all $t_1,t_2 \ge 0$;\\[1mm]
\null\hskip 8pt $\bullet$ For any $T > 0$, the map $(t,u) \mapsto 
\Phi(t)u$ is continuous from $[0,T]\times X$ to $X$.

\smallskip\noindent
In particular, given initial data $u_0 \in X$, the {\em trajectory} 
$u : \R_+ \to X$ defined by $u(t) = \Phi(t)u_0$ is a continuous function 
of the time $t \ge 0$,  and the solution $u(t)$ depends continuously 
on $u_0$, uniformly in time on compact intervals. If $\Phi(t)u_0 = u_0$ 
for all $t \ge 0$,  we say that $u_0$ is an {\em equilibrium}.

\begin{definition}\label{edsdef}
We say that a continuous semiflow $(\Phi (t))_{t\ge 0}$ on a metrizable 
space $X$ is an {\em extended dissipative system} on $\R$ if one can 
associate to each element $u\in X$ a triple $(e,f,d)$ with $e,d \in C^0(\R,
\R_+)$ and $f\in C^0(\R,\R)$ such that\\[1mm]
\null\hskip 8pt {\bf (A1)} The functions $e,f,d$ depend continuously on 
$u \in X$, uniformly on compact sets of $\R$;\\[1mm]
\null\hskip 8pt {\bf (A2)} $\,|f|^2 \le b(e)d$ for some nondecreasing 
function $b: \R_+ \to \R_+$;\\[2mm]
and such that, under the evolution defined by the 
semiflow $(\Phi(t))_{t\ge0}$, the time-dependent quantities 
$e,f,d$ have the following properties\:\\[2mm]
\null\hskip 8pt {\bf (A3)} If $d(x,t) = 0$ for all $(x,t) \in \R \times 
[0,t_0]$, where $t_0 > 0$, then $u$ is an equilibrium;\\[1mm]
\null\hskip 8pt {\bf (A4)} The energy balance $\partial_t e = 
\partial_x f - d$ holds in the sense of distributions on 
$\R \times \R_+$.
\end{definition}

\vskip 1pt
\begin{remark}\label{abuse}
As in \cite{GS2} there is a slight abuse of notation in the
definition above. To any state of the system, namely to any point $u
\in X$, we associate an energy density $e(x) \ge 0$, an energy flux
$f(x) \in \R$, and an energy dissipation rate $d(x) \ge 0$, which
are continuous functions of $x \in \R$ and satisfy properties (A1),
(A2). In addition, given any $t \ge 0$, we associate to the 
evolved state $\Phi(t)u \in X$ an energy density $e(x,t) \ge 0$, an 
energy flux $f(x,t) \in \R$, and an energy dissipation rate $
d(x,t) \ge 0$, and these are the time-dependent quantities that 
satisfy (A3), (A4). For simplicity, we will use the same notation 
$e,f,d$ in both cases, although the quantities that evolve according 
to the semiflow $\Phi(t)$ depend on an additional variable $t \ge 0$. 
\end{remark}

\begin{remark}\label{1D}
Unlike in \cite{GS2}, where definitions are given in full generality, 
we only consider here an extended dissipative system on the real 
line $\R$. This is because we want to study the Navier-Stokes 
equation in the cylinder $\O$, which has only one unbounded 
direction, so that we indeed obtain a one-dimensional extended 
dissipative system if we consider the energy flow through vertical 
sections of the cylinder, as in \eqref{edef}--\eqref{ddef}.
\end{remark} 

We next introduce a functional-analytic framework that is appropriate
for the Navier-Stokes equation \eqref{NSeq}.  Let $X$ denote the
Banach space
\begin{equation}\label{Xdef}
  X \,=\, \Bigl\{u \in \BUC(\O)\,\Big|\, \partial_j u \in \BUC(\O)
  \hbox{ for } j = 1,2\,,~ \div u = 0\,,~\langle u_1\rangle = 0\Bigr\}~,
\end{equation}
equipped with the norm $\|u\|_X \,=\, \max(\|u\|_{L^\infty(\O)},
\|\nabla u\|_{L^\infty(\O)})$. We recall that $\langle u_1\rangle$ is
the vertical average of the horizontal velocity, see \eqref{czero}.
Proceeding as in the proof of Theorem~\ref{thm1}, it is
straightforward to verify that the Cauchy problem for Eq.~\eqref{NSeq}
is globally well-posed in $X$. More precisely, for any $u_0 \in X$,
the integral equation \eqref{NSint} has a unique global solution $u
\in C^0([0,\infty),X)$, which depends continuously on the initial data
$u_0$ in the topology of $X$, uniformly in time on compact
intervals. In other words, the Navier-Stokes equation defines a
continuous semiflow $(\Phi (t))_{t\ge 0}$ on $X$.  Given a constant 
$M > 0$, we also consider the subset
\begin{equation}\label{XMdef}
  X_M \,=\, \Bigl\{u \in X\,\Big|\, |\curl u| = |\partial_1 u_2 - 
  \partial_2 u_1| \le M\Bigr\} \,\subset\, X~,
\end{equation}
which is invariant under the semiflow $\Phi(t)$ since the vorticity
$\omega = \curl u$ obeys the parabolic maximum principle. Note that
$X_M$ is an {\em unbounded} subset of $X$, because (as was discussed
in the previous section) a bound on the vorticity allows to control
the oscillating part $\widehat u$ of the velocity field, but not the
average vertical velocity $m = \langle u_2\rangle$.

For any $u \in X_M$, we define as in \eqref{edef}--\eqref{ddef}\:
\begin{align}\nonumber
  e(x_1) \,&=\, \frac12 \int_\T |u(x_1,x_2)|^2 \dd x_2 + 1~, \\ \label{defehd}
  h(x_1) \,&=\, \int_\T \Bigl(\frac12 |u(x_1,x_2)|^2 + 
  p(x_1,x_2)\Bigr)u_1(x_1,x_2)\dd x_2~, \\ \nonumber
  d(x_1) \,&=\, \int_\T |\nabla u(x_1,x_2)|^2 \dd x_2~,
\end{align}
where $p$ is given by \eqref{pdef2}. We also denote $f(x_1) = 
\partial_1 e(x_1) - h(x_1)$. 

\begin{lemma}\label{l:ehd}
Assume that $u \in X_M$ for some $M > 0$, and let $e,h,d$ be defined 
by \eqref{defehd}. Then there exists a constant $C_4 > 0$, 
depending only on $M$, such that
\begin{equation}\label{hebounds}
  h(x_1)^2 \,\le\, C_4 e(x_1) d(x_1)~, \qquad |\partial_1 e(x_1)|^2 
  \,\le\, 2 e(x_1) d(x_1)~,
\end{equation}
for all $x_1 \in \R$. 
\end{lemma}

\begin{proof}
If $u \in X_M$, then $\langle u_1 \rangle = \int_\T u_1(x_1,x_2)\dd x_2 
= 0$, and the Poincar\'e-Wirtinger inequality implies
\begin{equation}\label{PW}
  \int_\T |u_1(x_1,x_2)|^2\dd x_2 \,\le\, \frac{1}{4\pi^2}\int_\T 
  |\partial_2 u_1(x_1,x_2)|^2 \dd x_2 \,\le\, \frac{d(x_1)}{4\pi^2}~. 
\end{equation}
To prove \eqref{hebounds}, we write $h(x_1) = h_1(x_1) + h_2(x_1)$, 
where
\[
  h_1(x_1) \,=\, \frac12 \int_\T |u(x_1,x_2)|^2u_1(x_1,x_2) 
  \dd x_2~, \qquad 
  h_2(x_1) \,=\, \int_\T p(x_1,x_2)u_1(x_1,x_2)\dd x_2~. 
\]
Since $u_2 = m + \widehat u_2$, where $m = \langle u_2 \rangle$, we 
have $|u|^2 = u_1^2 + m^2 + 2m\widehat u_2 + \widehat u_2^2$, hence
\[
  h_1 \,=\, \frac12 \int_\T \Bigl(u_1^2 + 2m\widehat u_2 + 
  \widehat u_2^2\Bigr)u_1 \dd x_2 \,=\, \frac12 \int_\T \Bigl(u_1^2 
  - \widehat u_2^2 + 2u_2\widehat u_2\Bigr)u_1 \dd x_2~.
\]
Using Lemma~\ref{l:bounds}, H\"older's inequality, and Poincar\'e's 
inequality \eqref{PW}, we thus obtain
\begin{align*}
  |h_1| \,&\le\, CM \int_\T (|u_1| + |u_2|)|u_1| \dd x_2 \,\le\,  
  CM \Bigl(\int_\T |u|^2 \dd x_2\Bigr)^{1/2} \Bigl(\int_\T u_1^2\dd 
  x_2\Bigr)^{1/2} \,\le\, CM (ed)^{1/2}~, \\
  |h_2| \,&\le\, CM^2 \int_\T |u_1| \dd x_2 \,\le\,  CM^2 \Bigl(\int_\T 
  u_1^2 \dd x_2\Bigr)^{1/2} \,\le\, CM^2 d^{1/2} \,\le\, CM^2 (ed)^{1/2}~. 
\end{align*}
In the last inequality, we used the fact that $e(x_1) \ge 1$ for 
all $x_1 \in \R$. Finally, 
\[
  |\partial_1 e| \,\le\, \int_\T |u_1\partial_1 u_1 +  u_2\partial_1 
  u_2| \dd x_2 \,\le\, \Bigl(\int_\T |u|^2\dd x_2\Bigr)^{1/2}
  \Bigl(\int_\T |\nabla u|^2\dd x_2\Bigr)^{1/2}  \,\le\, (2ed)^{1/2}~.
\]
\end{proof}

As a direct consequence of Lemma~\ref{l:ehd}, we obtain

\begin{proposition}\label{NSeds}
The Navier-Stokes equations \eqref{NSeq}, \eqref{pdef} define an extended 
dissipative system in the space $X_M$ for any $M > 0$. More precisely,
the triple $(e,f,d)$ defined by \eqref{defehd} and by the relation 
$f = \partial_1 e - h$ satisfies assumptions (A1), (A3), (A4)
of Definition~\ref{edsdef}, as well as\\[1mm]
\null\hskip 8pt {\bf (A2')} $\,|f|^2 \le \beta ed$ for some positive 
  constant $\beta$ depending only on $M$;\\[1mm]
\null\hskip 8pt {\bf (A5)} $\,\,|\partial_1 e|^2 \le \gamma ed$ for 
  some positive constant $\gamma$. 
\end{proposition}

We are now in position to apply the results of \cite{GS2}. If 
$u(x_1,x_2,t)$ is a solution of \eqref{NSeq}, \eqref{pdef} with 
initial data $u_0 \in X_M$, we define for all $T \ge 0$\:
\begin{equation}\label{estardef}
  e_*(T) \,=\, \sup_{x_1 \in \R} e(x_1,T)~, \qquad
  \bar e_*(T) \,=\, \sup_{0 \le t \le T} e_*(t)~,
\end{equation}
where $e(x_1,t)$ is the energy density \eqref{edef}. With these 
notations we have

\begin{corollary}\label{NSedsres} {\bf \cite{GS2}}
If $u_0 \in X_M$ for some $M > 0$, the solution of \eqref{NSeq}, 
\eqref{pdef} with initial data $u_0$ satisfies, for all $x_1 \in \R$ and 
all $T > 0$, 
\begin{equation}\label{NSeds1}
  \Bigl|\int_0^T f(x_1,t)\dd t\Bigr| \,\le\, \sqrt{\beta T 
  \bar e_*(T)e_*(0)}~,
\end{equation}
where $\beta  > 0$ is as in Proposition~\ref{NSeds}. Moreover, 
for all $R > 0$, 
\begin{equation}\label{NSeds2}
  \int_0^T \int_{|x_1| \le R} d(x_1,t)\dd x_1 \dd t \,\le\, 
  2 \sqrt{\beta T\bar e_*(T)e_*(0)} + 2 R e_*(0)~.
\end{equation}
In particular, if $e_*(t)$ is uniformly bounded for all times, 
we have
\begin{equation}\label{NSeds3}
  \limsup_{T \to \infty} \frac{1}{\sqrt{\beta T}} \int_0^T \int_{|x_1| 
  \le \sqrt{\beta T}} d(x_1,t)\dd x_1 \dd t \,\le\, 4 \sqrt{e_*(0) 
  \bar e_*(\infty)}~.
\end{equation}
\end{corollary}

The weakness of Corollary~\ref{NSedsres} lies in the fact that
the right-hand side of inequalities \eqref{NSeds1}--\eqref{NSeds3}
involves the quantity $\bar e_*(T)$, which depends on $T$ in an
unknown way. As was discussed in the introduction, it is still an
open question whether $\bar e_*(T)$ is uniformly bounded in time for 
any solution of \eqref{NSeq} in $X_M$. If we restrict ourselves 
to solutions for which $\bar e_*(\infty) < \infty$, then 
Corollary~\ref{NSedsres} allows to draw interesting consequences
on the long-time behavior. For instance, inequality \eqref{NSeds1}
shows that the energy flux $f(x_1,t)$ through any fixed point 
$x_1 \in \R$ is, on average, very small when $t$ is large. It 
follows that the total energy that is dissipated in a spatial 
domain of size $\OO(T^{1/2})$ over the time interval $[0,T]$ grows
at most like $T^{1/2}$ as $T \to \infty$, as indicated by 
\eqref{NSeds2}, \eqref{NSeds3}. Since, by assumption (A3), the
energy dissipation $d(x_1,t)$ vanishes only on the set of equilibria,
one can deduce, as in \cite{GS2}, that any solution of \eqref{NSeq}
with uniformly bounded energy density converges, in a suitable 
localized and averaged sense, to the set of equilibria as 
$t \to \infty$. In particular, it follows from \eqref{NSeds2}
or \eqref{NSeds3} that the Navier-Stokes equation \eqref{NSeq} 
has no solutions in $\BUC(\O)$ that are periodic in time, 
except for spatially homogeneous equilibria. 

The main goal of the present paper is to reproduce the results of
\cite{GS2} for the Navier-Stokes equation without assuming that the
energy density stays uniformly bounded. As we shall show in the
subsequent sections, this can be achieved by using the additional
assumption (A5) in Proposition~\ref{NSeds}, which holds in our case
but not for some of the systems considered in \cite{GS2}. We shall
thus obtain estimates which are similar to \eqref{NSeds1}--\eqref{NSeds3}
but do not contain the quantity $\bar e_*(T)$ in the right-hand side. 

\begin{remark}\label{compactness}
If we equip our function space $X$ with a topology that is 
weak enough so that all solutions with uniformly bounded 
energy density are compact, then Corollary~\ref{NSedsres} also 
gives some information on the omega-limit set of such solutions. 
A natural choice is the ``localized'' topology $\TT_\loc$, 
which is the topology of uniform convergence on compact sets 
for the velocity field $u(x_1,x_2)$. Indeed, by Ascoli's 
theorem, any bounded subset of $X$ is relatively compact 
with respect to $\TT_\loc$. Moreover, although the Navier-Stokes
equation is nonlocal, it is straightforward to verify that
the solutions of \eqref{NSeq}, \eqref{pdef} depend continuously on the 
initial data in the topology $\TT_\loc$, so that \eqref{NSeq} 
defines a continuous semiflow in $X_\loc = (X,\TT_\loc)$. If 
we restrict ourselves to the subset $X_M$ for some $M > 0$, 
then \eqref{NSeq} nearly defines an extended dissipative 
system in the sense of Definition~\ref{edsdef}. The only 
caveat concerns assumption (A1): the flux $f$ and the 
dissipation $d$ do not depend continuously on $u$ in the localized 
topology $\TT_\loc$, but this is a minor point and most of the results 
of \cite{GS2} can be derived without that property. For instance, if 
$u(t)$ is a solution of \eqref{NSeq} that stays uniformly bounded in 
$X$, then the trajectory $\{u(t)\}_{t\ge0}$ is relatively 
compact in $X_\loc$, and it follows from \cite{GS2} that the
omega-limit set $\Omega(u)$ contains at least an equilibrium
$\bar u$ with $\nabla \bar u \equiv 0$. Moreover the solution
$u(t)$ stays most of the time, in a sense that can be 
quantified precisely, within an arbitrary neighborhood (in
$\TT_\loc$) of the set of spatially homogeneous equilibria. 
\end{remark}

\section{Integrated flux bounds}\label{sec4}

From now on, we do not consider specifically the Navier-Stokes
equation anymore, but we study a general extended dissipative 
system on $\R$ in the sense of Definition~\ref{edsdef}. Keeping
in mind the applications to \eqref{NSeq}, we strengthen assumption
(A2) as follows:\\[2mm]
\null\hskip 8pt {\bf (A2')} $\,|f|^2 \le \beta ed$ for some positive 
  constant $\beta$.\\[2mm]
Moreover, we add another hypothesis, which will be crucial
in obtaining results without a priori bounds on the solutions\:\\[2mm]
\null\hskip 8pt {\bf (A5)} $\,\,|\partial_x e|^2 \le \gamma ed$ for 
  some positive constant $\gamma$.\\[2mm]
Here and in what follows, we denote the space variable by $x \in \R$ 
(instead of $x_1$). Assumption (A5) means that the energy gradient 
generates dissipation, and in combination with (A2') this will drive the 
whole theory. 

Given a solution $u(t) = \Phi(t)u_0$ of our system, we consider 
the (time-dependent) energy density $e(x,t) \ge 0$, the energy 
flux $f(x,t) \in \R$, and the energy dissipation rate $d(x,t) \ge 0$, 
which are continuous functions on $\R \times \R_+$. In view 
of (A4), the local energy dissipation law $\partial_t e = \partial_x 
f - d$ holds in the sense of distributions on $\R \times \R_+$. 
As a consequence, given $T > 0$ and $a,b \in \R$ with $a < b$, 
we have the integrated energy balance equation
\begin{equation}\label{r:balance}
  \int_a^b\Bigl(e(x,T)-e(x,0)\Bigr)\dd x \,=\, 
  \int_0^T\Bigl(f(b,t)-f(a,t)\Bigr)\dd t - \int_0^T\int_a^b
  d(x,t)\dd x\dd t~.
\end{equation}
The left-hand side is the variation of energy in the segment $[a,b]$
from initial time $t = 0$ to final time $t = T$. The first term in the
right-hand side represents the energy entering (or leaving) the segment
$[a,b]$ through the endpoints over the time interval $[0,T]$, and the
last term is the energy that is dissipated in $[a,b]$ for $t \in
[0,T]$. This relation will be used so often that we now introduce
shorthand notations for the various quantities in \eqref{r:balance}.

We use capital letters $E,F,D$ to denote the integrals of $e,f,d$ 
with respect to time, namely
\begin{equation}\label{EFDdef}
  E(x,T) \,=\, \int_{0}^{T}e(x,t)\dd t~, ~\quad
  F(x,T) \,=\, \int_{0}^{T}f(x,t)\dd t~, ~\quad 
  D(x,T) \,=\, \int_{0}^{T}d(x,t)\dd t~, 
\end{equation}
for all $x \in \R$ and all $T \ge 0$. Thus, if $a < b$, the total energy 
dissipated in the segment $[a,b]$ over the time interval $[0,T]$ is
\begin{equation}\label{Ddef}
  D([a,b],T) \,=\, \int_a^b D(x,T) \dd x \,=\, \int_a^b\int_0^T d(x,t) 
  \dd t \dd x~.
 \end{equation}
Another important quantity is the "available" energy in the segment 
$[a,b]$ at time $T$, which we define as
\begin{equation}\label{Adef}
  A([a,b],T) \,=\, \int_a^b e(x,0)\dd x+ F(b,T) - F(a,T)~.
\end{equation}
This is the energy that would be present in the segment $[a,b]$ at
time $T$, due to the initial data and to the flux through the 
endpoints, if no dissipation was included in our model. Indeed, 
using this notation, the integrated energy balance \eqref{r:balance} 
reads
\begin{equation}\label{r:intbalance}
  A([a,b],T) \,=\, \int_a^b e(x,T)\dd x + D([a,b],T)~.  
\end{equation}
Finally, one of our main goals is investigating the energy growth, 
so it is convenient to introduce the following notations for the
supremum of the energy density: 
\begin{equation}\label{Estardef} 
  \begin{array}{l} \DS\,e_*(t) \,=\, \sup_{x\in \R}e(x,t)~, \\
  \DS E_*(t) \,=\, \sup_{x\in \R}E(x,t)~, \end{array} \qquad
  \begin{array}{l} \DS
  \DS \,e_*([a,b],t) \,=\, \sup_{x\in [a,b]}e(x,t)~, \\
  \DS E_*([a,b],t) \,=\, \sup_{x\in [a,b]}E(x,t)~. \end{array}
\end{equation}

As a first application, we use the energy balance equation and 
assumption (A2') to derive useful bounds on the integrated energy 
flux $F(x,T)$, which will serve as a basis for the analysis in the subsequent 
sections. We begin with a local version of the integrated flux bound. 

\begin{proposition}\label{p:localflux}
Let $u(t) = \Phi(t)u_0$ be any solution of an extended dissipative
system on $\R$ satisfying (A2'). Then, for all $a,b \in \R$ with
$a < b$ and all $T > 0$, we have
\begin{align}\label{r:left} 
  F(a,T) \,&\le\,  \phantom{+}\sqrt{\beta e_*([a,b],0)E_*([a,b],T)} \,+\,
  \frac{\beta E_*([a,b],T)}{b-a}~,  \\  \label{r:right}
  F(b,T) \,&\ge\, -\sqrt{\beta e_*([a,b],0)E_*([a,b],T)} \,-\, 
  \frac{\beta E_*([a,b],T)}{b-a}~. 
\end{align}
\end{proposition}

\begin{proof}
Let $\widetilde{e} = e_*([a,b],0)$ and $\widetilde{E} =E_*([a,b],T)$. If 
$\widetilde{E}=0$, then by (A2') we have $F(x,T) = 0$ for all $x\in [a,b]$, 
and \eqref{r:left}, \eqref{r:right} trivially hold. So we assume that 
$\widetilde{E}>0$ and prove \eqref{r:left}, the proof of \eqref{r:right} 
being analogous. Using assumption (A2') and H\"older's inequality, 
we find for any $x\in [a,b]$:
\begin{align}\nonumber
  F(x,T)^{2} \,&\le\, \Bigl(\int_0^T |f(x,t)| \dd t\Bigr)^2 \,\le\, \Bigl(
  \int_0^T \beta^{1/2}e(x,t)^{1/2} d(x,t)^{1/2}\dd t\Bigr)^2 \\ \label{F1}
  \,&\le\,\beta \Bigr(\int_0^T e(x,t)\dd t\Bigr) \Bigl(
  \int_0^Td(x,t)\dd t\Bigr) \,\le\, \beta \widetilde{E}\,D(x,T)~.
\end{align}
On the other hand, if we integrate in time the energy dissipation law (A4) and
use the fact that $e(x,T) \ge 0$ and $e(x,0) \le \widetilde{e}$, we obtain 
for all $x \in [a,b]$:
\begin{equation}\label{F2}
  \partial_x F(x,T) \,=\, e(x,T) - e(x,0) + D(x,T) \,\ge\, 
  -\widetilde{e}+D(x,T)~.
\end{equation}
Thus, combining \eqref{F1} and \eqref{F2}, we see that the integrated 
flux $F(x,t)$ satisfies the differential inequality
\begin{equation}\label{r:diffin}
  \partial_x F(x,T) \,\ge\, -\widetilde{e} +\frac{1}{\beta \widetilde{E}}
  \,F(x,T)^{2}~, \qquad x \in [a,b]~.  
\end{equation}

Let $\rho = (\beta \widetilde{e}\widetilde{E})^{1/2}$. If $F(a,T)\le \rho$, 
then \eqref{r:left} is proved. If $F(a,T) > \rho$, then $\partial_x F(a,T)>0$, 
and it follows from \eqref{r:diffin} that $F(x,T) >\rho$ for all 
$x\in [a,b]$, so that
\[
  \frac{\partial_x F(x,T)}{F(x,T)^2 - \rho^2} \,\ge\, 
  \frac{1}{\beta \widetilde{E}}~, \qquad x \in [a,b]~.  
\]
Integrating both sides over $x \in [a,b]$ we deduce
\[
  \frac{b-a}{\beta \widetilde{E}} \,\le\,\frac{1}{2 \rho} 
  \ln\left(\frac{F(a,T)+\rho}{F(a,T)-\rho} \cdot\frac{F(b,T)- 
  \rho}{F(b,T)+\rho}\right) \\ 
  \,\le\, \frac{1}{2 \rho} \ln\left(\frac{F(a,T)
  +\rho}{F(a,T)-\rho}\right)~.
\]
Thus, if we denote $Y = \rho(b-a)/(\beta\widetilde{E})$, we arrive at 
\[
  F(a,T) \,\le\, \rho \,\frac{\exp (2Y)+1}{\exp (2Y)-1} \,\le\, \rho 
  \Bigl(1+\frac{1}{Y}\Bigr) \,=\, \sqrt{\beta \widetilde{e}\widetilde{E}}
  +\frac{\beta \widetilde{E}}{b-a}~,
\]
which is the desired result. 
\end{proof}

A remarkable feature of inequalities \eqref{r:left}, \eqref{r:right}
is that they give estimates on the integrated flux $F(x_0,T)$ in terms
of the energy density $e(x,t)$ for $t \in [0,T]$ and $x$ in a {\em
neighborhood} of $x_0$. Simpler estimates involving the energy over
the whole line $\R$ easily follow from Proposition~\ref{p:localflux}
and are often sufficient in the applications.

\begin{corollary}\label{c:fluxbound} 
Under the assumptions of Proposition~\ref{p:localflux}, one has for all 
$x \in \R$ and all  $T>0$
\begin{equation}\label{fluxbound}
  |F(x,T)| \,\le\, \sqrt{\beta e_*(0)E_*(T)}~,  
\end{equation}
where $e_*(0)$ and $E_*(T)$ are defined in \eqref{Estardef}.   
\end{corollary}

\begin{proof}
Since $e_*([a,b],0) \le e_*(0)$ and $E_*([a,b],T) \le E_*(T)$, it follows from 
\eqref{r:left}, \eqref{r:right} that
\[
  F(a,T) \,\le\, \sqrt{\beta e_*(0)E_*(T)} +  \frac{\beta E_*(T)}{b-a}~, \qquad
  F(b,T) \,\ge\, -\sqrt{\beta e_*(0)E_*(T)} -  \frac{\beta E_*(T)}{b-a}~.
\]
If we now take the limit $b \to +\infty$ in the first inequality and $a \to 
-\infty$ in the second one, we obtain \eqref{fluxbound}. 
\end{proof}

The proof of Proposition~\ref{p:localflux} also shows that, in a
one-dimensional extended dissipative system, the energy density
$e(x,t)$ cannot be everywhere an increasing function of time. More
precisely, we have

\begin{corollary}\label{c:zero}
Let $u(t) = \Phi(t)u_0$ be a solution of an extended dissipative system 
on $\R$ satisfying (A2'), and assume that there exists $T > 0$ such 
that $e(x,T) \ge e(x,0)$ for all $x \in \R$. Then $f(x,t) = d(x,t) = 0$ 
and $e(x,t) = e(x,0)$ for all $x \in \R$ 
and all $t \in [0,T]$. 
\end{corollary}

\begin{proof}
If $e(x,T) \ge e(x,0)$ for all $x \in \R$, it is clear that inequality 
\eqref{F2} holds for all $x \in \R$ with $\widetilde{e} = 0$. Proceeding 
as in the proof of 
Proposition~\ref{p:localflux} and Corollary~\ref{c:fluxbound}, we deduce 
that $F(x,T) \le 0$ for all $x \in \R$, and finally that $F(\cdot,T) \equiv 0$. 
Since $e(x,T) \ge e(x,0)$, the integrated energy balance \eqref{r:balance}
then implies that the energy dissipation $d(x,t)$ vanishes identically for 
$t \in [0,T]$, and so does the energy flux $f(x,t)$ by (A2'). Now, using the 
local energy dissipation law (A4), we conclude that $e(x,t) = e(x,0)$ for
all $x \in \R$ and all $t \in [0,T]$.
\end{proof}

In view of assumption (A3), Corollary~\ref{c:zero} implies that
equilibria are the only possible solutions for which $e(x,t) \ge
e(x,0)$ for all $x \in \R$ and all $t \ge 0$.  In particular, a state
$u \in X$ for which the energy density $e(x)$ vanishes identically (or
is equal to its minimal value) is necessary an equilibrium.  For
instance, it follows from Corollary~\ref{c:zero} that the
Navier-Stokes equation \eqref{NSeq} has no nontrivial solution $u \in
C^0_b((0,T],X)$ in the space \eqref{Xdef} such that $u(x,t)$ converges
to zero uniformly on compact sets of $\R$ as $t \to 0$.

\section{Integrated energy bounds}\label{sec5}

We have seen in Corollary~\ref{c:fluxbound} that the integrated
energy flux $F(x,T)$ can be bounded by an expression depending only 
on the initial data and the integrated energy density $E_*(T)$. 
Using the additional assumption (A5), we now prove that $E_*(T)$ can 
in turn be estimated in terms of the initial data and the observation
time $T$. We begin with

\begin{lemma}\label{l:L2first}
Let $u(t) = \Phi(t)u_0$ be any solution of an extended dissipative
system on $\R$ satisfying (A5). Given $T > 0$ and $a,b \in \R$ with
$a < b$, we have for all $x \in [a,b]$\:
\begin{equation}\label{L2first}
  \left(\int_0^T e(x,t)^2\dd t\right)^{1/2} \le\,
  \biggl(\frac{\sqrt{T}}{b-a} + \sqrt{\gamma}\biggr) 
  \sup_{0 \le t \le T} A([a,b],t)~,
\end{equation}
where the available energy $A([a,b],t)$ is defined in \eqref{Adef}
and satisfies \eqref{r:intbalance}. 
\end{lemma}

\begin{proof}
For any $x_0 \in [a,b]$ and any $t \in [0,T]$, we have
\[
  \left|e(x_0,t) - \frac{1}{b-a}\int_a^b e(x,t)\dd x\right|
  \,\le\, \int_a^b|\partial_x e(x,t)|\dd x~.  
\]
Applying (A5) and H\"older's inequality, we obtain
\[
  \int_a^b |\partial_x e(x,t)|\dd x \,\le\, \sqrt{\gamma} 
  \int_a^b \Bigl(e(x,t)d(x,t)\Bigr)^{1/2}\dd x \,\le\, \sqrt{\gamma}
  \Bigl(\int_a^b e(x,t)\dd x\Bigr)^{1/2} \Bigl(\int_a^b d(x,t)\dd x
  \Bigr)^{1/2}~,
\]
hence
\begin{equation}\label{r:lone}
  e(x_0,t) \,\le\, \frac{1}{b-a}\int_a^b e(x,t)\dd x + 
  \sqrt{\gamma} \Bigl(\int_a^b e(x,t)\dd x\Bigr)^{1/2} 
  \Bigl(\int_a^b d(x,t)\dd x\Bigr)^{1/2}~.
\end{equation}

Using \eqref{r:lone} we now estimate the $L^2$ norm in time 
of $e(x_0,t)$. Since $\int_a^b e(x,t)\dd x \le A([a,b],t)$ by 
\eqref{r:intbalance}, we have
\begin{equation}\label{r:ltwo}
  \int_0^T \Bigl(\int_a^b e(x,t)\dd x\Bigr)^2\dd t \,\le\, 
  \int_0^T A([a,b],t)^2\dd t \,\le\, T A_*^2~,
\end{equation}
where we introduced the shorthand notation $A_* = \sup\{A([a,b],t)\,|\,
t \in [0,T]\}$. Similarly, in view of \eqref{Ddef} and 
\eqref{r:intbalance}, we obtain
\begin{equation}\label{r:lthree}
  \int_0^T \Bigl(\int_a^b e(x,t)\dd x\Bigr)\Bigl(\int_a^b d(x,t)\dd x
  \Bigr)\dd t \,\le\, \sup_{0 \le t \le T}\Bigl(\int_a^b e(x,t)
  \dd x\Bigr)D([a,b],T) \,\le\, A_*^2~.
\end{equation}
Applying Minkowski's inequality to the right-hand side of \eqref{r:lone} 
and using \eqref{r:ltwo}, \eqref{r:lthree}, we thus find
\[
  \Bigl(\int_0^T e(x_0,t)^2 \dd t\Bigr)^{1/2} \,\le\, 
  \frac{\sqrt{T} A_*}{b-a} + \sqrt{\gamma} A_*~,
\]
which is the desired result.
\end{proof}

\begin{remark}\label{r:power2}
It is clear from inequality \eqref{r:lone} that we can estimate
the $L^p$ norm in time of the energy density $e(x_0,t)$ for 
$p \le 2$ only. Indeed, the only control we have on the 
energy dissipation $d(x,t)$ is the bound $D([a,b],T) \le 
A([a,b],T)$, which comes from \eqref{r:intbalance}, and this 
corresponds to the limiting case $p = 2$. 
\end{remark}

In view of \eqref{L2first}, it is natural to introduce the quantity
\begin{equation}\label{EEstardef}
  \EE_*(T) \,=\, \sup_{x \in \R}  \left(\int_0^T e(x,t)^2\dd t
  \right)^{1/2}~, \qquad T > 0~,
\end{equation}
which controls $E_*(T)$ since $E_*(T) = \sup_{x\in\R}\int_0^T e(x,t)\dd t
\le \sqrt{T} \EE_*(T)$. Combining Corollary~\ref{c:fluxbound} and 
Lemma~\ref{l:L2first}, we obtain the main result of this section\:

\begin{proposition}\label{p:L2energy}
Let $u(t) = \Phi(t)u_0$ be any solution of an extended dissipative
system on $\R$ satisfying (A2') and (A5). There exists a constant 
$\kappa > 1$, depending only on the product $\beta\gamma$, such that, 
for all $T > 0$, 
\begin{equation}\label{L2energy}
  \EE_*(T) \,\le\, \kappa e_*(0)\sqrt{T}~, \qquad \hbox{and}\qquad
  E_*(T) \,\le\, \kappa e_*(0) T~.
\end{equation}
\end{proposition}

\begin{proof}
We need only prove the first inequality in \eqref{L2energy}, which 
implies the second one. Fix $T > 0$, and take $a,b \in \R$ with 
$a < b$. By \eqref{Adef} and \eqref{fluxbound}, the available 
energy in the interval $[a,b]$ at any time $t \in [0,T]$ 
satisfies 
\[
  A([a,b],t) \,\le\, e_*(0)(b-a) + |F(b,t)| + |F(a,t)| 
  \,\le\, e_*(0)(b-a) + 2\sqrt{\beta e_*(0)E_*(t)}~.
\]
If $x \in [a,b]$, it thus follows from \eqref{L2first} that
\begin{equation}\label{L2second}
  \Bigl(\int_0^T e(x,t)^2\dd t\Bigr)^{1/2} \le\,
  \Bigl(\frac{\sqrt{T}}{b-a} + \sqrt{\gamma}\Bigr) 
  \Bigl(e_*(0)(b-a) + 2\sqrt{\beta e_*(0)E_*(T)}\Bigr)~,
\end{equation}
because $E_*(t) \le E_*(T)$ if $t \in [0,T]$. We now assume that 
$b-a = \epsilon \sqrt{T}$, where $\epsilon = (\beta/\gamma)^{1/4}$. 
Inserting this relation into the right-hand side of \eqref{L2second}, 
and then taking the supremum over $x \in \R$ in the left-hand side, 
we obtain the inequality
\begin{equation}\label{L2third}
  \EE_*(T) \,\le\, (1+\sigma)\Bigl(e_*(0)\sqrt{T} + 2\sigma
  \sqrt{e_*(0)E_*(T)}\Bigr)~,
\end{equation}
where $\sigma = (\beta\gamma)^{1/4}$. If $e_*(0) = 0$, then $\EE_*(T) 
= 0$ in agreement with \eqref{L2energy}. In the converse case, we 
define $Z > 0$ such that
\[
  Z^2 \,=\, \frac{\EE_*(T)}{e_*(0) \sqrt{T}}~,
\]
and since $E_*(T) \le \sqrt{T}\EE_*(T)$ we deduce from \eqref{L2third}
that $Z^2 \le (1+\sigma)(1+2\sigma Z)$. This quadratic inequality
implies that $Z^2 \le \kappa$, where
\begin{equation}\label{kappadef}
  \sqrt{\kappa} \,=\, \sigma(1+\sigma) + \sqrt{\sigma^2
  (1+\sigma)^2 + (1+\sigma)}~,
\end{equation}
and \eqref{L2energy} follows. This concludes the proof.
\end{proof}

As an immediate consequence of Corollary~\ref{c:fluxbound} and 
Proposition~\ref{p:L2energy}, we obtain our final estimate 
on the integrated energy flux\:

\begin{corollary}\label{c:flux}
Under the assumptions of Proposition~\ref{p:L2energy} we 
have, for all $x \in \R$ and all $T > 0$,
\begin{equation}\label{r:flux}
  |F(x,T)| \,\le\, e_*(0)\sqrt{\kappa\beta T}~.  
\end{equation}
\end{corollary}

\section{Some dynamical implications}\label{sec6}

In this section, we draw a few consequences of the previous results,
in the spirit of what was done in \cite{GS2} for {\em bounded}
solutions of extended dissipative systems. As in
Proposition~\ref{p:L2energy}, we always assume that $u(t) =
\Phi(t)u_0$ is a solution of an extended dissipative system on $\R$
satisfying (A2') and (A5).  We first observe that our bound on the
integrated energy flux implies a useful estimate on the energy
dissipation.

\begin{proposition}\label{p:dissip}
Under the assumptions of Proposition~\ref{p:L2energy} we have, 
for all $T > 0$ and all $R > 0$,  
\begin{equation}\label{dissip}
  \int_{-R}^R e(x,T) \dd x + \int_0^T \int_{-R}^R d(x,t)\dd x\dd t
  \,\le\, 2 e_*(0)\Bigl(R + \sqrt{\kappa \beta T}\Bigr)~,
\end{equation}
where $\kappa$ is defined in \eqref{kappadef}.      
\end{proposition}

\begin{proof}
By \eqref{r:intbalance} the left-hand side of \eqref{dissip} is 
equal to the available energy $A([-R,R],T)$. Now, using definition
\eqref{Adef} and Corollary~\ref{c:flux}, we see that $A([-R,R],T)
\le 2R e_*(0) + 2e_*(0)\sqrt{\kappa\beta T}$, and \eqref{dissip} follows.  
\end{proof}

Inequality \eqref{dissip} shows that the dissipated energy
$D([-R,R],T)$ grows at most like $\sqrt{T}$ as $T \to \infty$.  In
particular, all equilibria of our extended dissipative system are
non-dissipative (i.e., they satisfy $d \equiv 0$), and there exist no
other time-periodic solutions. Moreover, since by assumption (A3) only
equilibria satisfy $d \equiv 0$, Proposition~\ref{p:dissip} can be
used to prove that all trajectories converge, in a suitable sense, to
the set of equilibria as $t \to \infty$. For instance, arguing as in
\cite[Proposition~5.1]{GS2}, we obtain

\begin{corollary}\label{c:conveq1}
Consider an extended system on $\R$ satisfying (A2') and (A5). 
If $\bar u \in X$ is not an equilibrium, then $\bar u$ has a 
neighborhood $\VV$ in $X$ such that, for any solution $u(t) = 
\Phi(t)u_0$, one has
\[
  \limsup_{T \to \infty} \frac{1}{\sqrt{T}}\int_0^T \1_\VV
  (u(t))\dd t \,<\, \infty~,
\]
where $\1_\VV$ denotes the characteristic function of $\VV$. 
\end{corollary}

Corollary~\ref{c:conveq1} shows that any trajectory $u(t)$ spends a
very small fraction of its lifetime in a sufficiently small
neighborhood of any nonequilibrium point. If we assume in addition
that our configuration space $X$ is {\em compact} (see
Remark~\ref{compactness}), then using a finite covering argument we
can deduce that any trajectory spends most of its time near the set of
equilibria. More precisely, proceeding as in
\cite[Proposition~5.4]{GS2}, we find

\begin{corollary}\label{c:conveq2}
Consider a compact extended system on $\R$ satisfying (A2') and (A5). 
If $\VV$ is a neighborhood of the set of equilibria, then any 
solution $u(t) = \Phi(t)u_0$ satisfies 
\[
  \limsup_{T \to \infty} \frac{1}{\sqrt{T}}\int_0^T \1_{\VV^c}
  (u(t))\dd t \,<\, \infty~,
\]
where $\1_{\VV^c}$ denotes the characteristic function of 
$X \setminus \VV$. 
\end{corollary}

\begin{remark}\label{measures}
Corollary 6.3. has several ergodic-theoretical implications for
compact extended systems satisfying (A2') and (A5). For instance, 
one can show by applying Birkhoff's ergodic theorem that all 
invariant measures are supported on the set of equilibria. 
Furthermore, using the variational principle for topological and 
metric entropy, one can conclude that the topological entropy of the 
system is necessarily zero, see \cite[Section 4]{Sl} for a 
related discussion. 
\end{remark}

We now drop the compactness assumption and return to the general
case considered in Proposition~\ref{p:L2energy}. We have already
observed in Corollary~\ref{c:zero} that only trivial solutions
(namely, equilibria) have the property that the energy density
$e(x,t)$ does not decrease anywhere in space when times varies. 
We now derive a more precise result which strongly constraints 
the set of points where energy can increase. Given a solution 
$u(t) = \Phi(t)u_0$ of an extended dissipative system on $\R$, 
we define, for any $T > 0$,
\begin{equation}\label{JRdef}
  J_T \,=\, \biggl\{R>0\,\bigg|~\int_{-R}^R e(x,T)\dd x  \,\ge\, 
  \int_{-R}^Re(x,0)\dd x\biggr\} \,\subset\, (0,\infty)~.
\end{equation}

\begin{proposition}\label{p:finiteLeb}
Let $u(t) = \Phi(t)u_0$ be any solution of an extended dissipative
system on $\R$ satisfying (A2') and (A5), and assume that $u_0 \in X$
is not an equilibrium. Then, for any $T > 0$, the set $J_T$ defined 
by \eqref{JRdef} has a finite Lebesgue measure.
\end{proposition}

\begin{proof}
Given $T > 0$, we define for any $R > 0$
\[
  \partial E(R,T) \,=\, \int_{-R}^R e(x,T)\dd x - \int_{-R}^R 
  e(x,0)\dd x~.
\]
The energy balance \eqref{r:balance} then implies
\begin{equation}\label{r:FDF}
  F(R,T) - F(-R,T) \,=\, \partial E(R,T) + D([-R,R],T)~, \qquad
  R > 0~.
\end{equation}
On the other hand, proceeding as in \eqref{F1} and using 
\eqref{L2energy}, we find for all $x \in \R$
\[
  F(x,T)^2 \,\le\, \beta \Bigl(\int_0^T e(x,t)\dd t\Bigr) \Bigl(
  \int_0^T d(x,t)\dd t\Bigr) \,\le\, \beta \kappa e_*(0)T \,D(x,T)~,
\]
so that $\int_{-R}^R F(x,T)^2\dd x \le \beta \kappa e_*(0)T \,D([-R,R],T)$.
In view of Corollary~\ref{c:zero}, the assumption that $u_0$ is not an 
equilibrium implies that $e_*(0) > 0$, hence we deduce from \eqref{r:FDF}
that
\begin{equation}\label{r:FEF}
  F(R,T)-F(-R,T) \,\ge\, \partial E(R,T) + \frac{1}{\beta \kappa e_*(0)T}
  \int_{-R}^RF(x,T)^2\dd x~. 
\end{equation}

Since $u_0\in X$ is not an equilibrium, assumption (A3) implies
that $D([-R,R],T)>0$ for all sufficiently large $R > 0$, 
say for all $R \ge R_0 > 0$. On the other hand, by definition, 
we have $\partial E(R,T) \ge 0$ for all $R\in J_T$. Thus, 
using \eqref{r:FDF}, we conclude that $F(R,T)-F(-R,T)>0$ for all 
$R\in J_T\cap [R_0,\infty)$. 

If $J_T\cap \lbrack R_0,\infty )$ is empty, the claim is proved.
Otherwise, we choose $R_1\in J_T\cap \lbrack R_0,\infty)$, and 
we define
\[
  \FF(R) \,=\, \frac{1}{2(\beta \kappa e_*(0)T)^2}
  \int_{-R}^R F(x,T)^2 \dd x~, \qquad R > 0~.
\]
The function $\FF:(0,\infty )\to \R_+$ is nondecreasing and $\FF(R)>0$ 
for all $R \ge R_1$. Using \eqref{r:FEF} and the definition of $J_T$, 
we obtain
\begin{align*}
  \FF'(R) \,&=\, \frac{F(R,T)^2 + F(-R,T)^2}{2(\beta \kappa e_*(0)T)^2}
  \,\ge\, \frac{|F(R,T) - F(-R,T)|^2}{4(\beta \kappa e_*(0)T)^2} \\
  \,&\ge\, \frac{\1_{J_T}(R)}{4(\beta \kappa e_*(0)T)^4}
  \biggl(\int_{-R}^RF(x,T)^2\dd x\biggr)^2 \,=\, \1_{J_T}(R)\FF(R)^2~.
\end{align*}
Thus, for all $R > R_1$, we have
\begin{equation*}
  \int_{R_1}^R \1_{J_T}(R)\dd R \,\le\, \int_{R_1}^R \frac{\FF'(R)}{
  \FF(R)^2}\dd R \,\le\, \frac{1}{\FF(R_1)}-\frac{1}{\FF(R)}
  \,\le\, \frac{1}{\FF(R_1)}~.
\end{equation*}
This proves that $J_T\cap \lbrack R_1,\infty)$ has finite 
Lebesgue measure. 
\end{proof}

\section{Pointwise estimates on the energy density}\label{sec7}

In Sections~\ref{sec4} and \ref{sec5} we have shown that, under
assumptions (A2') and (A5), the energy density associated to any
solution of an extended dissipative system on $\R$ satisfies nice
integral bounds, which are summarized in
Proposition~\ref{p:L2energy}. A more difficult question is whether our
hypotheses also imply a uniform estimate in time on the energy
density. This is an important open problem, which we hope to address
in a future work. Before giving a partial result in that direction, we
observe that some naive blow-up scenarios are already excluded by
Proposition~\ref{p:L2energy} and Proposition~\ref{p:dissip}.  For
instance, if for some $x \in \R$ the energy density $e(x,t)$ is a
nondecreasing function of time, then \eqref{L2energy} implies that
$e(x,t) \le \kappa e_*(0)$ for all $t \ge 0$. Indeed, for any $T > t$ 
we have
\[
  e(x,t) \,\le\, \frac{1}{T-t}\int_t^T e(x,\tau)\dd\tau \,\le\,
  \frac{E_*(T)}{T-t} \,\le\, \frac{\kappa e_*(0)T}{T-t}~, 
\]
and the claim follows by taking $T \to \infty$. Thus a standard
scenario where the maximum of the energy density is reached at a fixed
point $x \in \R$ and increases with time cannot lead to any unbounded
growth.  On the other hand, applying \eqref{dissip} with $R =
\sqrt{\beta T}$, we obtain
\[
  \frac{1}{\sqrt{\beta T}}\int_{-\sqrt{\beta T}}^{\sqrt{\beta T}} 
  \,e(x,T)\dd x  \,\le\, 2(1+\sqrt{\kappa})e_*(0)~.  
\]
Thus, if for some $T > 0$ the energy density $e(x,T)$ is comparable to
$e_*(T)$ over an interval of size $2\sqrt{\beta T}$, then $e_*(T)$ is
in turn comparable to $e_*(0)$. This indicates that strong spatial
inhomogeneities necessarily occur in unbounded solutions, if they
exist.

To obtain a pointwise bound on the energy density in the abstract 
framework of extended dissipative systems, it appears necessary 
to introduce an additional assumption, which allows to control 
the spatial derivative of $e(x,t)$ at a given time. A reasonable
possibility is\:\\[1mm]
\null\hskip 8pt {\bf (A6)} $\,(\partial_x e)^2 \le \delta e$ for some
$\delta > 0$.\\[1mm]
Of course, (A6) follows from (A5) if the energy dissipation rate
$d(x,t)$ is uniformly bounded, which is indeed the case in many 
applications. Under this hypothesis, we have the following result. 

\begin{proposition}\label{p:growth}
Let $u(t) = \Phi(t)u_0$ be any solution of an extended dissipative
system on $\R$ satisfying (A2'), (A5), and (A6). There exists a 
constant $C > 0$, depending only on $\beta\gamma$, such that, for 
all $T > 0$, 
\begin{equation}\label{growth}
  e_*(T) \,\le\, C\Bigl(e_*(0) + e_*(0)^{2/3}(\delta\beta T)^{1/3}
  \Bigr)~. 
\end{equation}
\end{proposition}

\begin{proof}
We fix $T > 0$ and assume that $e_*(T) > 0$ (otherwise there is 
nothing to prove). Given $x_0 \in \R$, we have either $e(x_0,T) 
\le 4e_*(0)$, or $e(x_0,T) > 4e_*(0)$. In the latter case, 
we define $a = e(x_0,T)/\sqrt{\delta e_*(T)}$. Since 
$|\partial_x e(x,T)| \le \sqrt{\delta e_*(T)}$ by (A6), we have 
for all $x \in \R$
\[
  e(x,T) \,\ge\, e(x_0,T) - \sqrt{\delta e_*(T)}\,|x-x_0| \,=\, 
  e(x_0,T)\Bigl(1-\frac{|x-x_0|}{a}\Bigr)~.
\]
Thus, by \eqref{r:intbalance}, the available energy in the interval 
$[x_0-a,x_0+a]$ at time $T$ satisfies
\begin{equation}\label{pow1}
  A([x_0-a,x_0+a],T) \,\ge\, \int_{x_0-a}^{x_0+a} e(x,T)\dd x 
  \,\ge\, a e(x_0,T)~.
\end{equation}
On the other hand, in view of \eqref{Adef} and Corollary~\ref{c:flux}, 
we also have
\begin{equation}\label{pow2}
  A([x_0-a,x_0+a],T) \,\le\, 2ae_*(0) + 2 e_*(0)\sqrt{\kappa \beta T}~.
\end{equation}
Combining \eqref{pow1}, \eqref{pow2} and recalling that $e(x_0,T) 
> 4e_*(0)$, we thus find
\begin{equation}\label{pow3}
  a e(x_0,T) \,\equiv\, \frac{e(x_0,T)^2}{\sqrt{\delta e_*(T)}} \,\le\, 
  4 e_*(0)\sqrt{\kappa\beta T}~.
\end{equation}
Summarizing, given $x_0 \in \R$, we have shown that \eqref{pow3} holds
whenever $e(x_0,T) > 4e_*(0)$. Since $e_*(T) = \sup_{x_0 \in \R}
e(x_0,T)$, we conclude that
\[
  e_*(T) \,\le\, \max\Bigl(4 e_*(0)\,,\, (4e_*(0))^{2/3} (\delta\kappa
  \beta T)^{1/3}\Bigr)~,
\]
and \eqref{growth} follows. 
\end{proof}

In the particular case of the Navier-Stokes equation \eqref{NSeq}, it
is possible to use Proposition~\ref{p:growth} to prove
Theorem~\ref{thm2}, but this approach requires a uniform bound on the
energy dissipation rate $d$ which is not obvious a priori. In fact, if
$u$ is any solution of \eqref{NSeq} in the space $X$ defined by
\eqref{Xdef}, we know from \eqref{nablauest} that $\nabla u$ is
uniformly bounded in the space $\BMO(\O)$ for all $t \ge 0$, because
the vorticity $\omega = \partial_1 u_2 - \partial_2 u_1$ is bounded in
$L^\infty(\O)$, but this is not sufficient to control the energy
dissipation \eqref{ddef} in $L^\infty(\R)$.  However, using the
vorticity equation \eqref{omeq} and the fact that the only possibly
unbounded component of the velocity field is the vertical average $m =
\langle u_2\rangle$, it is possible to prove that the vorticity
$\omega$ is uniformly bounded in some H\"older space $C^\alpha(\O)$
for all $t \ge t_0 > 0$, and using the Biot-Savart formula we can
deduce that $\nabla u$ is uniformly bounded in $L^\infty(\O)$ for all
$t \ge 0$. This implies that the energy dissipation $d(x_1,t)$ is
bounded in $L^\infty(\R)$, so that (A6) follows from (A2'), and
Proposition~\ref{p:growth} allows us to conclude that the energy
density $e(x_1,t)$ cannot grow faster than $t^{1/3}$ as $t \to
\infty$, which shows \eqref{ubound2}.

Alternatively, Theorem~\ref{thm2} can be established by the following
direct argument,  which does not rely on Proposition~\ref{p:growth}. 

\medskip\noindent{\it Proof of Theorem~\ref{thm2}.} Let $u(x,t)$ be a
solution of the Navier-Stokes equations \eqref{NSeq}, \eqref{pdef}
given by Theorem~\ref{thm1}. Since we are interested in the long-time
behavior, we can assume without loss of generality that the initial
data $u_0$ belong to the set $X_M$ for some $M > 0$, see definition
\eqref{XMdef} and Proposition~\ref{localex}. Also, we suppose that $u$
is decomposed as in \eqref{udecomp} with $(m_1,m_2) = (0,m)$. We
already know that, for all $t \ge 0$,
\begin{equation}\label{omubounds}
  \|\omega(\cdot,t)\|_{L^\infty(\O)} \,\le\, M~, \qquad \hbox{and}
  \qquad\|\widehat u(\cdot,t)\|_{L^\infty(\O)} \,\le\, C_1 M~, \qquad
\end{equation}
see Lemma~\ref{l:bounds}. Since $\partial_1 m = \langle \omega
\rangle$, it follows that $|\partial_1 m(x_1,t)| \le M$ for all 
$x_1 \in \R$ and all $t \ge 0$. 

In the subsequent calculations, we fix a time $t > 0$ and, for 
simplicity, we denote the space variable by $x$ instead of 
$x_1$. Given $a > 0$, we have for all $x \in \R$\:
\begin{equation}\label{mbound1}
  |m(x,t)| \,\le\, \frac{Ma}{2} + \frac{1}{2a}\int_{x-a}^{x+a}|m(y,t)| 
  \dd y~,
\end{equation}
because
\[
  \frac{1}{2a}\int_{x-a}^{x+a} \Bigl(|m(x,t)| - |m(y,t)|\Bigr) \dd y 
  \,\le\, \frac{1}{2a}\int_{x-a}^{x+a} M|x-y| \dd y \,=\, \frac{Ma}{2}~.
\]
To bound the last term in \eqref{mbound1}, we observe that $|u|^2 = 
\widehat u_1^2 + (m + \widehat u_2)^2 \ge \frac12 m^2  - 
\widehat u_2^2$. Integrating that inequality with respect to the 
vertical variable and using \eqref{edef}, \eqref{omubounds},  we 
easily obtain
\begin{equation}\label{mbound2}
  m(x,t)^2 \,\le\, 4 e(x,t) + C~, \qquad x \in \R~,
\end{equation}
where $C = 2 C_1^2 M^2$. Thus, if we apply H\"older's inequality 
to the integral in \eqref{mbound1} and use \eqref{mbound2}, we 
arrive at 
\begin{equation}\label{mbound3}
  |m(x,t)| \,\le\, \frac{Ma}{2} + \Bigl( \frac{2}{a}\int_{x-a}^{x+a}
  e(y,t) \dd y + C\Bigr)^{1/2}~, \qquad x \in \R~.
\end{equation}
Finally, we know from Lemma~\ref{l:ehd} and Proposition~\ref{NSeds}
that the Navier-Stokes equation in $X_M$ defines an extended dissipative 
system satisfying (A2') for some $\beta > 0$ (depending on $M$) 
and (A5) with $\gamma = 2$. Thus, proceeding as in 
Proposition~\ref{p:dissip}, we find 
\[
  \frac{2}{a}\int_{x-a}^{x+a}e(y,t) \dd y \,\le\, \frac{2}{a}\Bigl(2a e_*(0) 
   + 2e_*(0)\sqrt{\kappa\beta t}\Bigr) \,=\, 4e_*(0)\Bigl(1 + 
   \frac{\sqrt{\kappa\beta t}}{a}\Bigr)~,
\]
for all $x \in \R$. After replacing this inequality in the right-hand side
of \eqref{mbound3} and taking the supremum over $x \in \R$, 
we conclude that
\begin{equation}\label{mbound4}
  \sup_{x \in \R}|m(x,t)| \,\le\, \frac{Ma}{2} + \biggl(C + 4e_*(0)
  \Bigl(1 + \frac{\sqrt{\kappa\beta t}}{a}\Bigr)\biggr)^{1/2}~.
\end{equation}
If we now take $a = t^{1/6}$, we see from \eqref{mbound4} that 
$\|m(\cdot,t)\|_{L^\infty} \le C'(1+t)^{1/6}$ for some $C' > 0$. 
Since $\widehat u$ is uniformly bounded, this proves \eqref{ubound2}. 
\QED

\begin{remark}\label{r:pointwise} It follows from the above proof 
(after optimizing the choice of $a$) that any solution of the 
Navier-Stokes equations \eqref{NSeq}, \eqref{pdef} with initial 
data in $X_M$ satisfies
\[
  \limsup_{t \to +\infty}\frac{\|u(\cdot,t)\|_{L^\infty}}{t^{1/6}}
  \,\le\, K e_*(0)^{1/3}~,
\]
where $K = 3(M/2)^{1/3}(\beta\kappa)^{1/6}$ depends only on $M$.
\end{remark}

\section{Convergence results for the Navier-Stokes 
equations}\label{sec8}

This final section is entirely devoted to the particular example of
the Navier-Stokes equations in the cylinder $\O = \R \times \T$. Our
goal is to use the results of Sections~\ref{sec4} to \ref{sec6} to
obtain qualitative informations on the long-time behavior of the
solutions. Without loss of generality, we fix $M > 0$ and work in 
the function space $X_M$ defined by \eqref{XMdef}, where equations 
\eqref{NSeq}, \eqref{pdef} define an extended dissipative system 
satisfying assumptions (A2'), (A5) for some constants $\beta, 
\gamma$. In particular, applying Proposition~\ref{p:L2energy} and 
using the explicit formula \eqref{edef} for the energy density, we 
obtain the estimate
\[
  \sup_{x_1 \in \R}\, \int_0^T \!\!\int_\T \frac12 |u(x_1,x_2,t)|^2 
  \dd x_2 \dd t \,\le\, \kappa e_*(0) T~,
\]
which proves \eqref{ubound3}. Similarly, if we denote $B_R = 
[-R,R] \times \T$, then Proposition~\ref{p:dissip} implies that
\begin{equation}\label{speedbound}
  \int_{B_R} \frac12 |u(x,T)|^2 \dd x + \int_0^T \int_{B_R} 
  |\nabla u(x,t)|^2 \dd x \dd t  \,\le\, 2e_*(0)(R + 
  \sqrt{\kappa\beta T})~,
\end{equation}
which is \eqref{ubound4}. Thus, to complete the proof of
Theorem~\ref{thm3}, it remains to establish \eqref{ubound5}. 

Let $\EE \subset X$ denote the set of equilibria of the 
Navier-Stokes equation \eqref{NSeq} in $X$, namely the set
of all constant velocity fields of the form $u = (0,m)^\tsp$, 
with $m \in \R$. Given $u \in X$ and $R > 0$, we define the 
distance from $u$ to $\EE$ on the finite cylinder $B_R = [-R,R] 
\times \T$ as 
\begin{equation}\label{dRdef}
  d_R(u,\EE) \,=\, \inf_{m \in \R} \,\sup_{x \in B_R} |u(x) - 
  (0,m)^\tsp|~.
\end{equation}
The following estimate will be useful\:

\begin{lemma}\label{l:dist}
Fix $\theta \in (0,1)$. There exists $C_5 > 0$ such 
that, for any $u \in X_M$ and any $R \ge 1$, one has
\begin{equation}\label{distest}
  d_R(u,\EE) \,\le\, C_5 M^\theta R^{\frac{1+\theta}{2}} \|\nabla 
  u\|_{L^2(B_R)}^{1-\theta}~.
\end{equation}
\end{lemma}

\begin{proof}
We decompose $u = (0,m)^\tsp + \widehat{u}$, where $m = \langle u_2
\rangle$. Since $\partial_1 m = \langle \omega\rangle$ and 
$|\omega| \le M$, we have using H\"older's inequality
\begin{equation}\label{lem1}
  \sup_{|x_1| \le R} |m(x_1) - m(0)| \,\le\, \int_{B_R} |\omega|\dd x \,\le\, 
  M^\theta \int_{B_R} |\omega|^{1-\theta}\dd x \,\le\, C 
  M^\theta R^{\frac{1+\theta}{2}} \|\omega\|_{L^2(B_R)}^{1-\theta}~.
\end{equation}
On the other hand, since $\widehat{u}$ has zero mean over $B_R$, 
the Sobolev embedding theorem and the Poincar\'e-Wirtinger 
inequality imply that, if $2 < p < \infty$, 
\[
  \|\widehat{u}\|_{L^\infty(B_R)} \,\le\, C_p R^{1-\frac{1}{p}} \|\nabla 
  \widehat{u}\|_{L^p(B_R)}~,
\]
where the constant depends only on $p$ (here we use the 
assumption that $R \ge 1$). Moreover, interpolating between
$L^2$ and $\BMO$ and using \eqref{nablauest}, we obtain
\[
  \|\nabla \widehat{u}\|_{L^p(B_R)} \,\le\, C_p \|\nabla 
  \widehat{u}\|_{L^2(B_R)}^\frac{2}{p} \|\nabla \widehat{u}
  \|_{\BMO(\O)}^{1 - \frac{2}{p}} \,\le\, C_p M^{1 - \frac{2}{p}}
  \|\nabla \widehat{u}\|_{L^2(B_R)}^\frac{2}{p}~.
\]
Choosing $p = 2/(1{-}\theta)$, we thus find
\begin{equation}\label{lem2}
  \|\widehat{u}\|_{L^\infty(B_R)} \,\le\, C M^\theta R^{\frac{1+\theta}{2}}
  \|\nabla \widehat{u}\|_{L^2(B_R)}^{1-\theta}~.
\end{equation}
If we now combine \eqref{lem1} and \eqref{lem2}, we obtain
\[
  d_R(u,\EE) \,\le\, \sup_{x \in B_R} |u(x) - (0,m(0))^\tsp| \,\le\,  
  C_5 M^\theta R^{\frac{1+\theta}{2}}\|\nabla u\|_{L^2(B_R)}^{1-\theta}~,
\]
where $C_5 > 0$ depends only on $\theta$. This is the desired estimate. 
\end{proof}

The distance \eqref{dRdef} allows us to introduce the following 
family of neighborhoods of the set of equilibria. Given $\epsilon > 0$ 
and $R > 0$, we denote
\[
  \UU_{\epsilon ,R} \,=\, \{u\in X \,|\, d_R(u,\EE) < \epsilon\}~.
\]
Using estimate \eqref{speedbound} and Lemma~\ref{l:dist}, we 
now show that any solution of the Navier-Stokes equation in 
$X_M$ spends a relatively small fraction of its lifetime outside 
$\UU_{\epsilon,R}$, even if $\epsilon > 0$ is very small and 
$R > 0$ very large. More precisely, we have

\begin{proposition}\label{p:timespent}
Fix $\theta \in (0,1)$ and $M > 0$. There exists $C_6 > 0$ 
such that, if $u \in C^0([0,\infty),X)$ is any solution of the 
Navier-Stokes equations \eqref{NSeq}, \eqref{pdef} with 
initial data in $X_M$, the following estimate holds 
for any $\epsilon > 0$, any $R \ge 1$, and any $T > 0$\:
\begin{equation}\label{timespent}
  \int_0^T \1_{\UU_{\epsilon ,R}^c}(u(t)) \dd t \,\le\, \frac{C_6}{\epsilon}
  \,(RT)^{\frac{1+\theta}{2}}\Bigl(e_*(0)(R + \sqrt{\kappa\beta T})
  \Bigr)^{\frac{1-\theta}{2}}~,
\end{equation}
where $\1_{\UU_{\epsilon ,R}^c}$ is the characteristic function
of the complement of $\UU_{\epsilon ,R}$. 
\end{proposition}

\begin{proof}
Using the definition of the set $\UU_{\epsilon ,R}$ and 
estimate \eqref{distest}, we easily find
\[
 \int_0^T \1_{\UU_{\epsilon,R}^c}(u(t))\dd t \,\le\, 
 \frac{1}{\epsilon}\int_0^T d_R(u(t),\EE)\dd t 
 \,\le\, \frac{C_5}{\epsilon}\,M^\theta R^{\frac{1+\theta}{2}} 
 \int_0^T \|\nabla u(t)\|_{L^2(B_R)}^{1-\theta}\dd t~.
\]
Moreover, H\"older's inequality and estimate \eqref{speedbound}
imply
\[
  \int_0^T \|\nabla u(t)\|_{L^2(B_R)}^{1-\theta}\dd t \,\le\, 
  \biggl(\int_0^T \|\nabla u(t)\|_{L^2(B_R)}^2\dd t\biggr)^{\frac{1-
  \theta}{2}}T^{\frac{1+\theta}{2}} \,\le\, \Bigl(2e_*(0)(R + 
  \sqrt{\kappa\beta T})\Bigr)^{\frac{1-\theta}{2}}T^{\frac{1+\theta}{2}}~.
\]
Combining both inequalities, we arrive at \eqref{timespent}. 
\end{proof}

There are several ways to exploit the conclusion of 
Proposition~\ref{p:timespent}. If we fix $\epsilon, R$ and 
take $\theta$ sufficiently small, we obtain the following 
result which already implies estimate \eqref{ubound5} in 
Theorem~\ref{thm3}. 

\begin{corollary}\label{c:time1}
Any solution $u \in C^0([0,\infty),X)$ of the Navier-Stokes 
equations \eqref{NSeq}, \eqref{pdef} with initial data in 
$X_M$ satisfies
\[
  \limsup_{T\to \infty}\frac{1}{T^{\frac{3+\theta}{4}}}\int_0^T 
  \1_{\UU_{\epsilon,R}^{c}}(u(t))\dd t \,<\, \infty~,
\]
for all $\epsilon > 0$, all $R \ge 1$, and all $\theta > 0$. 
\end{corollary}

It is also interesting to consider a time-dependent domain $B_{R(T)}$
whose size increases (sufficiently slowly) as $T \to \infty$. In that
case, we can still show that any solution of \eqref{NSeq} converges 
to the set of equilibria inside $B_{R(T)}$, except perhaps on a sparse 
subset of the time axis.

\begin{corollary}\label{c:time2}
Fix $a,b,c>0$ such that $a/2+b<c<1/4$. If $u \in C^0([0,\infty),X)$ 
is any solution of the Navier-Stokes equations \eqref{NSeq}, 
\eqref{pdef} with initial data in $X_M$, there exists $C_7 > 0$ 
such that, for all $T \ge 1$, 
\begin{equation}\label{r:measure}
 \mathrm{meas}\biggl\{t\in [0,T]\,\bigg|\, \inf_{m\in \mathbf{R}}
 \sup_{|x_1|\le T^a}\sup_{x_2\in \T}|u(x_1,x_2)-(0,m)^\tsp|
 \ge \frac{1}{T^b}\biggr\} \,\le\, C_7 T^{\frac34+c}~.
\end{equation}
\end{corollary}

\begin{proof}
If we set $R = T^a$ and $\epsilon = T^{-b}$, the quantity in the
left-hand side of \eqref{r:measure} is exactly the integral 
$\int_0^T \1_{\UU_{\epsilon,R}^{c}}(u(t))\dd t$. Using \eqref{timespent} 
and the fact that $R = T^a \le \sqrt{T}$ since $a < 1/2$ and 
$T \ge 1$, we easily obtain
\[
  \int_0^T \1_{\UU_{\epsilon,R}^{c}}(u(t))\dd t \,\le\, 
  C (e_*(0))^{\frac{1-\theta}{2}} T^{\frac34 + \frac{a}{2} + b + 
  \frac{\theta}{4}(1+2a)}~.
\]
The conclusion now follows if we take $\theta > 0$ small enough.
\end{proof}

\section{Appendix}\label{sec9}

\subsection{Proof of Lemma~\ref{nlinlem}}
We start from \eqref{nlinexp} and recall that $\P_{jk} = \delta_{jk} 
+ R_j R_k$, where $R_1, R_2$ are the Riesz transforms. It follows 
that $\nabla\cdot e^{t\Delta}\P(u\otimes v) = \nabla\cdot e^{t\Delta}
(u\otimes v) + W(t,u,v)$, where 
\[
  W_j(t,u,v) \,=\, \sum_{k,\ell = 1}^2 \partial_\ell \,e^{t\Delta}
  R_j R_k u_\ell v_k \,=\, \sum_{k,\ell = 1}^2 \int_t^\infty \partial_j 
  \partial_k \partial_\ell \,e^{s\Delta} u_\ell v_k \dd s~, \qquad
  j = 1,2~.
\]
In the last equality, we used the fact that $R_j R_k \Delta = -\partial_j
\partial_k$ for $j,k = 1,2$. Now, for any $g \in L^\infty(\O)$ and 
any $t > 0$, we know that $e^{t\Delta}g \in C^\infty(\O)$ and 
$\|\partial^\alpha e^{t\Delta}g\|_{L^\infty} \le C_\alpha t^{-|\alpha|/2}
\|g\|_{L^\infty}$ for all $\alpha = (\alpha_1,\alpha_2) \in \N^2$ 
with $|\alpha| = |\alpha_1| + |\alpha_2| > 0$. If $u,v \in \BUC(\O)$, 
we thus have
\[
  \|\nabla\cdot e^{t\Delta}\P(u\otimes v)\|_{L^\infty} \,\le\, 
  \frac{C}{\sqrt{t}}\|u\|_{L^\infty}\|v\|_{L^\infty} + \int_t^\infty
  \frac{C}{s^{3/2}} \|u\|_{L^\infty}\|v\|_{L^\infty} \,\le\, 
  \frac{C}{\sqrt{t}}\|u\|_{L^\infty}\|v\|_{L^\infty}~,
\]
which proves \eqref{nlinest}. The same argument shows that
$\nabla\cdot e^{t\Delta}\P(u\otimes v) \in \BUC(\O)$. \QED

\subsection{Proof of Proposition~\ref{localex}}
Given $u_0 \in \BUC(\O)$ with $\div u_0 = 0$, we take $R > 0$ and $T >
0$ such that $2\|u_0\|_{L^\infty} \le R$ and $4 C_0 R T^{1/2} < 1$,
where $C_0 > 0$ is the constant in Lemma~\ref{nlinlem}. We introduce the
Banach space $X = C^0([0,T],\BUC(\O))$ equipped with the norm
$$
  \|u\|_X \,=\, \sup_{0 \le t \le T} \|u(t)\|_{L^\infty(\O)}~,
$$
and we set $B_R = \{u \in X\,|\, \|u\|_X \le R\}$. For all $u \in X$
and all $t \in [0,T]$, we denote by $(Fu)(t)$ the expression in 
the right-hand side of \eqref{NSint}. 

If $u \in B_R$, then by Lemma~\ref{nlinlem} we have for all 
$t \in [0,T]$\:
\begin{align*}
  \|(Fu)(t)\|_{L^\infty} \,\le\, \|u_0\|_{L^\infty} + 
  \int_0^t \frac{C_0}{\sqrt{t-s}}\,\|u(s)\|_{L^\infty}^2 \dd s
   \,\le\, \|u_0\|_{L^\infty} + 2 C_0 T^{1/2} R^2 \,\le\, R~.
\end{align*}
Thus $F$ maps $B_R$ into itself, and a similar calculation shows 
that $\|Fu - Fv\|_X \le \kappa \|u-v\|_X$ for all $u,v \in B_R$, 
where $\kappa = 4 C_0 R T^{1/2} < 1$. Thus Eq.~\eqref{NSint} has a 
unique solution in $B_R$, and applying Gronwall's lemma it is easy 
to verify that $u$ is also the unique solution of \eqref{NSint} in 
the whole space $X$. Finally, proceeding as in \cite{GMS}, one
can prove that $t^{1/2}\nabla u \in C^0_b((0,T],\BUC(\O))$. \QED

\subsection{Proof of Lemma \protect\ref{l:bounds}}

We first observe that, since $\widehat \omega$ has zero average
in the vertical variable, the Biot-Savart formula \eqref{BS} can 
be written in the equivalent form $\widehat u = \nabla^\perp 
\overline{K} * \widehat \omega$, where
$$
  \overline{K}(x_1,x_2) \,=\, K(x_1,x_2) - \frac{|x_1|}{2}~, \qquad
  (x_1,x_2) \in \O~.
$$
Now it is easy to verify that $\overline{K} \in L^1(\O)$ and 
$\partial_j \overline{K} \in L^1(\O)$ for $j = 1,2$, see \cite{AM}. 
Using Young's inequality, we deduce
\begin{equation}\label{uprel}
  \|\widehat u_1\|_{L^\infty} \,\le\, \|\partial_2 K\|_{L^1}
  \|\widehat \omega\|_{L^\infty}~, \qquad
  \|\widehat u_2\|_{L^\infty} \,\le\, \|\partial_1 \overline{K}\|_{L^1}
  \|\widehat \omega\|_{L^\infty}~,
\end{equation}
and the first inequality in \eqref{upest} follows since 
$\|\widehat \omega\|_{L^\infty} \le 2 \|\omega\|_{L^\infty}$. 

The next step is to establish the formula \eqref{pdef2} for 
the pressure. The easiest way is to use the identity
\begin{equation}\label{uid}
  \div((u\cdot\nabla)u) \,=\, \Delta(u_1^2) + 2 \partial_2
  (\omega u_1)~,
\end{equation}
which holds for any divergence-free vector field $u = (u_1,u_2)$ 
with vorticity $\omega = \partial_1 u_2 - \partial_2 u_1$. Since 
$-\Delta p = \div((u\cdot\nabla)u)$ and since $K$ is the fundamental 
solution of the Laplace operator, it follows from \eqref{uid}
that we can indeed take $-p = u_1^2 + 2\partial_2 K*(\omega u_1)$. 
It is also possible to derive \eqref{pdef2} directly from the formal 
expression \eqref{pdef}. Anyway, using \eqref{pdef2}, \eqref{uprel},
and the fact that $u_1 = \widehat{u}_1$, we find
\[
  \|p\|_{L^\infty} \,\le\, \|\widehat u_1\|_{L^\infty}^2 + 2 
  \|\partial_2 K\|_{L^1} \|\omega\|_{L^\infty}\|\widehat u_1\|_{L^\infty}
  \,\le\, C\|\partial_2 K\|_{L^1}^2 \|\omega\|_{L^\infty}^2~.
\]
This proves the second inequality in \eqref{upest}. 

Finally, to estimate $\nabla u$, we observe that
\[
  \partial_1 u_1 \,=\, -\partial_2 u_2 \,=\, R_1 R_2 \omega~,
  \quad \partial_1 u_2 = -R_1^2 \omega~, \quad \partial_2 u_1 = 
  R_2^2 \omega~,
\]
and we use the well-known fact that the Riesz operators are 
bounded from $L^\infty(\O)$ to $\BMO(\O)$, see \cite[Chapter IV]{St}. 
This concludes the proof of Lemma~\ref{l:bounds}. \QED

\end{document}